\documentclass[11pt,twoside, a4paper, english, reqno]{amsart}
\usepackage[dvips]{epsfig}
\usepackage{amscd}
\usepackage{amssymb}
\usepackage{amsthm}
\usepackage{amsmath}
\usepackage{latexsym}

\usepackage{upref}
\usepackage{hyperref}
\usepackage{color}
\usepackage{mathrsfs}
\usepackage[usenames,dvipsnames]{xcolor}

\theoremstyle{plain}
\newtheorem{thm}{Theorem}[section]
\theoremstyle{plain}
\newtheorem{lem}[thm]{Lemma}

\theoremstyle{definition}
\newtheorem{defi}{Definition}[section]

\newtheorem*{maintheorem*}{Main Theorem}

\newenvironment{Assumptions}
{
\setcounter{enumi}{0}

\begin{enumerate}}
{\end{enumerate} }

\newcommand{\R}{\ensuremath{\mathbb{R}}}

\newcommand{\goto}{\ensuremath{\rightarrow}}
\newcommand{\grad}{\ensuremath{\nabla}}

\numberwithin{equation}{section} \allowdisplaybreaks

\title[Space-Time Discretization of the Stochastic Allen-Cahn Equation]{ Optimal Strong Rates of Convergence for a Space-Time Discretization of
the Stochastic Allen-Cahn Equation with multiplicative noise}

\date{}



\author[Ananta K.\ Majee]{Ananta K.\ Majee}
\address[Ananta K.\ Majee]{\newline
Mathematisches Institut
Universit\"{a}t T\"{u}bingen
Auf der Morgenstelle 10
D-72076 T\"{u}bingen, Germany }
\email[]{majee@na.uni-tuebingen.de}

\author[Andreas Prohl]{ Andreas Prohl}
\address[Andreas Prohl]{\newline
Mathematisches Institut
Universit\"{a}t T\"{u}bingen
Auf der Morgenstelle 10
D-72076 T\"{u}bingen, Germany}
\email[]{dunst@na.uni-tuebingen.de}

\thanks{}

\begin{document}
\begin{abstract}
The stochastic Allen-Cahn equation with multiplicative noise involves
the nonlinear drift operator ${\mathscr A}(x) = \Delta x - \bigl(\vert x\vert^2 -1\bigr)x$. We use the fact that ${\mathscr A}(x) = -{\mathcal J}^{\prime}(x)$
satisfies a weak monotonicity property to deduce uniform bounds
in strong norms for solutions of the temporal, as well as of the spatio-temporal discretization of the problem. This weak monotonicity property then allows
for the estimate $ \underset{1 \leq j \leq J}\sup {\mathbb E}\bigl[ \Vert X_{t_j} - Y^j\Vert_{{\mathbb L}^2}^2\bigr] \leq C_{\delta}(k^{1-\delta} + h^2)$ for all small $\delta>0$, where
$X$ is the strong variational solution of the stochastic Allen-Cahn equation, while $\big\{Y^j:0\le j\le J\big\}$ solves a structure preserving finite element based space-time discretization
of the problem on a temporal mesh $\{ t_j;\, 1 \leq j \leq J\}$ of size $k>0$ which covers $[0,T]$.
\end{abstract}

\maketitle

\section{Introduction}
 Let $\bigl( {\mathbb H}, (\cdot, \cdot)_{{\mathbb H}}\bigr)$ be a separable Hilbert space, and ${\mathbb V}$ be a reflexive
Banach space, such that ${\mathbb V} \hookrightarrow {\mathbb H} \hookrightarrow {\mathbb V}'$ constitutes a Gelfand triple. The main motivation for
this work is to identify the structural properties for the drift operator of the nonlinear SPDE
\begin{equation}\label{eq:spde}
{\rm d}X_t = {\mathscr A}(X_t){\rm d}t + \sigma(X_t){\rm d}W_t \quad (t > 0)\,, \qquad 
X(0) = x \in {\mathbb H}\,,
\end{equation}
which allow to construct a space-time discretization of \eqref{eq:spde} for which optimal strong rates of convergence may be shown.
Relevant works in this direction are \cite{GM0,GM1}, where both, $\sigma$ and ${\mathscr A}$ are required to be Lipschitz. The Lipschitz assumption for the drift
operator ${\mathscr A}: {\mathbb V} \rightarrow {\mathbb V}'$ does not hold for many nonlinear SPDEs including the stochastic Navier-Stokes equation, or
the stochastic version of general phase field models (including \eqref{eq:Allen-Cahn}) below for example. A usual strategy for a related numerical analysis is then to truncate nonlinearities (see e.g.~\cite{PR0}), or to quantify the 
mean square error on large subsets $\Omega_{k,h} := \Omega_k \cap \Omega_h \subset \Omega$. As an example, the following estimate for a (time-implicit, finite element based) space-time 
discretization of the 2D stochastic Navier-Stokes equation with solution
$\{ {\bf U}^m;\, m \geq 0\}$ was obtained in \cite{CP1},
\begin{equation*}
{\mathbb E}\Bigl[ \chi_{\Omega_{k,h}}  \max_{1 \leq m \leq M} \Vert {\bf u}(t_m)- {\bf U}^m\Vert^2_{{\mathbb L}^2}
\Bigr] \leq C \bigl( k^{\eta - \varepsilon} + k h^{-\varepsilon} + h^{2-\varepsilon}\bigr) \qquad (\varepsilon > 0)
\end{equation*}
for all $\eta \in (0,\frac{1}{2})$, where $\Omega_k \subset \Omega$ (resp.~$\Omega_h \subset \Omega$) is such that ${\mathbb P}[\Omega \setminus \Omega_k] \rightarrow 0$ 
for $k \rightarrow 0$ (resp.~${\mathbb P}[\Omega \setminus \Omega_h] \rightarrow 0$ for $h \rightarrow 0$). We also mention the work \cite{KLM1} which studies a spatial discretization of the
stochastic Cahn-Hilliard equation. 
\vspace{.2cm}

Let ${\mathscr O} \subset {\mathbb R}^d$, $d \in \{1,2,3\}$ be a bounded Lipschitz domain. We consider the stochastic Allen-Cahn equation with multiplicative noise, where
the process $X: \Omega \times [0,T] \times \overline{\mathscr{O}} \rightarrow {\mathbb R}$ solves 
\begin{equation}\label{eq:Allen-Cahn}
{\rm d}X_t - \Bigl(\Delta X_t - \bigl(\vert X_t\vert^2-1\bigr) X_t\Bigr){\rm d}t  = \sigma(X_t){\rm d}W_t \quad (t > 0)\,, 
\qquad X_0 = x,
\end{equation}
where $W\equiv \{W_t; 0\le t\le T\}$ is an ${\mathbb R}$-valued Wiener process which is defined on the given filtered probability
space ${\mathfrak P} \equiv (\Omega, {\mathcal F}, {\mathbb F}, {\mathbb P})$; however it is easily possible to generalize the analysis below to a
trace class 
$Q$- Wiener process. Obviously, the drift operator ${\mathscr A}(y) = \Delta y - (\vert y\vert^2-1)y$ is only locally Lipschitz, but is the negative
 G\^{a}teaux differential of $\mathcal{J}(y) = \frac{1}{2} \Vert \nabla y\Vert^2_{{\mathbb L}^2} + \frac{1}{4} \||y|^2 - 1\|^2_{{\mathbb L}^2}$ and satisfies the weak monotonicity property
 \begin{align}
 \bigl\langle {\mathscr A}(y_1) - {\mathscr A}(y_2), y_1- y_2\bigr\rangle_{({\mathbb W}^{1,2}_{{\rm per}})^* \times {\mathbb W}^{1,2}_{{\rm per}}} \leq 
 K\, \|y_1 - y_2 \|^2_{{\mathbb L}^2} - \|\grad(y_1-y_2)\|_{\mathbb{L}^2}^2 \,\,\, \forall\, y_1, y_2 \in {\mathbb W}^{1,2}_{{\rm per}} \label{AD1}
 \end{align}
 for some $K > 0$; see Section~\ref{sec:tech-framework} for the notation. Our goal is a (variational) error analysis for the structure preserving finite element based space-time discretization
 \eqref{eq:finite-element-discretization} which accounts for this structural property, avoiding arguments that exploit only the locally Lipschitz property of $\mathscr{A}$ to arrive at optimal strong 
 error estimate. 
 \vspace{.1cm}
 
 The existing literature (see e.g.~\cite{KLM1}) for estimating the numerical strong error on the problem \eqref{eq:Allen-Cahn} mainly uses the involved linear semigroup theory; the authors have
 considered the additive colored noise case, in which they have benefited from it by using the stochastic convolution, and then used a
 truncation of the nonlinear drift operator $\mathscr{A}$ to prove a {\em rate of convergence} for a (spatial) semi-discretization
 on sets of probability close to $1$ without exploiting the weak monotonicity of $\mathscr{A}$. In contrast, property \eqref{AD1} and variational arguments were used in 
  the recent work \cite{prohl}, where strong error estimates for both, semi-discrete (in time) and fully-discrete schemes for \eqref{eq:Allen-Cahn} were obtained, which are of sub-optimal 
 order $\mathcal{O}(\sqrt{k} + h^\frac{2-\delta}{6})$ for the fully discrete scheme in the case $d=3$. In \cite{prohl}, a standard implicit discretization of \eqref{eq:Allen-Cahn} was considered for which it is not 
 clear to obtain uniform bounds for arbitrary higher moments of the solution of the fully discrete scheme, thus leading to sub-optimal convergence rates above. In this work, we consider
 the modified scheme \eqref{eq:finite-element-discretization} for \eqref{eq:Allen-Cahn}, and derive optimal strong numerical error estimate. 
 \vspace{.1cm}
 
 The subsequent analysis for the scheme \eqref{eq:finite-element-discretization} is split into two steps to independently address errors due to the temporal and spatial discretization. First we 
  exploit the variational solution concept for \eqref{eq:Allen-Cahn} and the semi-linear structure of ${\mathscr A}(y) = -{\mathcal J}^{\prime}(y)$
  to derive uniform bounds for the arbitrarily higher moments of the solution of \eqref{eq:Allen-Cahn} in strong norms; these bounds may then be used to bound increments of the solution of 
  \eqref{eq:Allen-Cahn} in Lemma~\ref{lem:esti-time-difference-strong-solun}. The second ingradient to achieve optimal error bounds is a temporal discretization which inherits the structural properties of 
  \eqref{eq:Allen-Cahn}; the scheme \eqref{eq:time-discrete} is constructed to allow for bounds of arbitrary moments of $\{ \mathcal{J}(X^j); 0\le j\le J\}$ in Lemma ~\ref{lem:moment-time-discrete}, which then 
  settles the error bounds in Theorem~\ref{thm:error-discrete in time} by using property \eqref{AD1} to effectively handle the nonlinear terms.
  We recover the asymptotic rate $\frac{1}{2}$ which is known for SPDEs of the form \eqref{eq:spde} when ${\mathscr A}$ is linear elliptic.
It is interesting to compare the present error analysis for the SPDE \eqref{eq:Allen-Cahn} with the one in \cite{HMS1} for a general SODE with polynomial
drift (see \cite[Ass.s 3.1, 4.1, 4.2]{HMS1}) which also exploits the weak monotonicity of the drift.
\vspace{.1cm}

The temporal semi-discretization was studied as a first step rather than spatial discretization to inherit bounds in strong norms which are needed for a complete error analysis of the problem. 
The second part of the error analysis is then on the structure preserving finite element based fully discrete scheme \eqref{eq:finite-element-discretization}, for which we first verify the
uniform bounds of arbitrary moments of $\{\mathcal{J}(Y^j); 0\le j\le J\}$ (cf.~ Lemma \ref{lem:moment-space-time-discrete}). It is worth mentioning that, if $\big\{Y^j:0\le j\le J\}$ is
a solution to a {\em standard} space-time  
 discretization which involves the nonlinearity $\mathscr{A}(Y^j)=-{\mathcal J}^{\prime}(Y^j)$, then only basic uniform
 bounds may be obtained (see \cite[Lemma $2.5$]{prohl}), as opposed to those in Lemma~\ref{lem:moment-space-time-discrete}. Next to it, we use again \eqref{AD1} for the drift, in combination with well-known approximation results 
 for a finite element discretization to show that the error part due to spatial discretization is of order $\mathcal{O}(\sqrt{k} + h)$
 where $k>0$ is the time discretization parameter and $h>0$ is the space discretization parameter (see Theorem~\ref{thm:error-estime}). In this context, we mention the numerical analysis
 in \cite{SS1} for an extended model of \eqref{eq:Allen-Cahn}, where the uniform bounds for the
exponential moments next to arbitrary moments in stronger norms are obtained for the solution of a semi-discretization in space
in the case $d=1$ (see \cite[Prop.s 4.2, 4.3]{SS1}); those bounds, together with a monotonicity argument are then used to properly address the nonlinear
effects in the error analysis and arrive at the (lower) strong rate $\frac{1}{2}$ for the $p$-th mean convergence of the numerical solution.
\vspace{.1cm}
% 
%  The rest of the paper is organized as follows. In Section~\ref{sec:tech-framework}, we detailed the technical framework and state the main result. Section \ref{sec:time-discrete} is devoted 
%  to show the strong convergence of the iterates $\{ X^j;\, j \in {\mathbb N}\}$
%  which solve \eqref{eq:time-discrete}; see Theorems~\ref{thm:error-discrete in time}; by using the higher moment estimates and the temporal bounds of the strong solution $X$,
%  which is detailed in Section \ref{sec:continuous-case}. In the final section, we prove the main theorem by employing uniform bounds for the higher order moments of the energy functional
%  for the solution of fully discrete scheme \eqref{eq:finite-element-discretization}, together with the error estimate of the solution of semi-discrete scheme and fully discrete scheme. 
 %%%%%%%%%%%%%%%%%%%%%%%%%%%%%%%%%%%%%%%%%%%%%%%%%%%%%%%%%%%%%%%%%%%%%%%%%%%%%%%%%%%%%%%%%%%%%%%%%%%%%%%%%%%%%%%%%%5
 \section{Technical framework and main result} \label{sec:tech-framework}
 Throughout this paper, we use the letter $C > 0$ to denote various generic constants. Let ${\mathscr O} \equiv (0,R)^d$, $1 \leq d \leq 3$, with $R \in (0,\infty)$ be
a cube in ${\mathbb R}^d$. Let us denote $\Gamma_j = \partial {\mathscr O} \cap \{ x_j = 0\}$ and
$\Gamma_{j + d} = \partial {\mathscr O} \cap \{x_j = R \}$ for $j = 1, \ldots, d$. The problem \eqref{eq:Allen-Cahn} is then supplemented
by the space-periodic boundary condition
$$ X \bigl\vert_{\Gamma_j} = X \bigl\vert_{\Gamma_{j+d}} \qquad (1 \leq j \leq d)\, .$$
Let $\bigl({\mathbb L}^p_{\rm per}, \Vert \cdot \Vert_{{\mathbb L}^{p}}\bigr)$ resp.~$\bigl({\mathbb W}^{m,p}_{\rm per}, \Vert \cdot \Vert_{{\mathbb W}^{m,p}}\bigr)$ denote the Lebesgue resp.~Sobolev space of
$R$-periodic functions $\varphi \in {\mathbb W}^{m,p}_{\rm loc}({\mathbb R}^d)$. Recall that functions in ${\mathbb W}^{m,2}_{{\rm per}}$ may be characterized by
their Fourier series expansion, i.e.,
\begin{eqnarray*}
{\mathbb W}^{m,2}_{{\rm per}}({\mathcal O}) &=& \Bigl\{ \varphi: {\mathbb R}^d \rightarrow {\mathbb R}: \varphi(x) =
\sum_{k \in {\mathbb Z}^d } c_k \exp \Bigl( 2i \pi \frac{\langle k, x\rangle}{R}\Bigr)\,, \\
&& \qquad \qquad  \overline{c}_k = c_{-k}\,, \ \sum_{k \in {\mathbb Z}^d } \vert k\vert^{2m}
\vert c_k\vert^2 < \infty \Bigr\}\, .
\end{eqnarray*}
Below, we set $\psi(x) = \frac{1}{4} \| |x|^2-1\|_{\mathbb{L}^2}^2$ for $x\in \mathbb{L}^2$. Throughout this article, we make the following assumption on $\sigma: \R\goto \R$.

\begin{Assumptions}
 \item \label{A1} $\sigma(0)=0$, and $\sigma,\, \sigma^{\prime},\, \sigma^{\prime\prime}$ are bounded. Moreover, $\sigma$ is Lipschitz continuous, i.e., there exists a constant $K_1 > 0$ such that
\begin{equation}\label{sac-1b}
 \Vert \sigma(u) - \sigma(v)\Vert^2_{{\mathbb L}^2} \leq \, K_1\, \Vert u-v\Vert^2_{{\mathbb L}^2}
\quad\forall\, u,v \in {\mathbb L}^2_{\rm per}\, .
\end{equation}
\end{Assumptions}
%%%%%%%%%%%%%%%%%%%%%%%%%%%%%%%%%%%%%%%%%%%%%%%%%%%%%%%%%%%%%%%%%%%%%%%%%%%%%%%%%%%%%%% 
\begin{defi}(Strong variational solution)
Fix $T \in (0,\infty)$, and $x \in {\mathbb L}^2_{\rm per}$. A $\mathbb{W}^{1,2}_{\rm per}$-valued ${\mathbb F}$-adapted stochastic process $X \equiv \{ X_t;\, t \in [0,T]\}$ is called a strong
variational solution of \eqref{eq:Allen-Cahn}
if $X \in L^2\bigl( \Omega; C([0,T]; {\mathbb L}^{2}_{\rm per})\bigr)$ satisfies ${\mathbb P}$-a.s.~for all $t \in [0,T]$ that
\begin{align}
& \big( X_t, \phi\big)_{\mathbb{L}^2} + \int_0^t \Big\{\big( \nabla X_s, \nabla \phi\big)_{\mathbb{L}^2} + \big( D\psi(X_s), \phi\big)_{\mathbb{L}^2}\Big\} \, {\rm d}s \notag \\
 &\quad = \big( x, \phi\big)_{\mathbb{L}^2} + \int_0^t\big( \sigma(X_s), \phi\big)_{\mathbb{L}^2}\, {\rm d}W_s \quad \forall\, \phi \in {\mathbb W}^{1,2}_{\rm per}. \label{eq:variational-Allen-Cahn}
\end{align}
\end{defi}

The following estimate for the strong solution is well-known ($p \ge 1$),
\begin{equation}\label{esti:higher-moment-in-l2-strong-solun}
\sup_{t \in [0,T]} {\mathbb E}\Bigl[ \frac{1}{p} \Bigl(\Vert X_t\Vert^p_{{\mathbb L}^2} - \Vert x \Vert^p_{{\mathbb L}^2}\Bigr) + \int_0^T \Vert X_s\Vert^{p-2}_{{\mathbb L}^2} 
\Bigl(\Vert \nabla X_s\Vert^2_{{\mathbb L}^2} + \Vert X_s\Vert^4_{{\mathbb L}^4}\Bigr)\, {\rm d}s
\Bigr] \leq C\, .
\end{equation}

\subsection{Fully discrete scheme} 
Let us introduce some notation needed to define the structure preserving finite element based fully discrete scheme. Let $0 = t_0 < t_1 < \ldots < t_J$ be a uniform partition
of $[0,T]$ of size $k = \frac{T}{J}$. Let ${\mathscr T}_h$ be a quasi-uniform triangulation of the domain ${\mathscr O}$. We
consider the ${\mathbb W}^{1,2}_{{\rm per}}$-conforming finite element space (cf.~\cite{BS1}) $\mathbb{V}_h\subset {\mathbb W}^{1,2}_{{\rm per}} $ such that 
$$ {\mathbb V}_h = \bigl\{ \phi \in C(\overline{\mathscr{O}};\R);\, \phi \bigl\vert_K \in {\mathscr P}_1(K) \quad
\forall\, K \in {\mathscr T}_h\bigr\},$$  where $\mathscr{P}_1(K)$ is the space of $\R$-valued functions on $K$ which are polynomials of degree less or
equal to $1$. We may then consider the space-time discretization of \eqref{eq:Allen-Cahn}: Let $Y^0 = {\mathscr P}_{{\mathbb L}^2}x \in {\mathbb V}_h$,
where ${\mathscr P}_{{\mathbb L}^2}: {\mathbb L}^2_{{\rm per}} \rightarrow {\mathbb V}_{h}$ denotes the ${\mathbb L}^2_{{\rm per}}$-orthogonal
projection, i.e., for all $g\in {\mathbb L}^{2}_{{\rm per}}$ 
\begin{align}
 \big( g- \mathscr{P}_{\mathbb{L}^2}g, \phi \big)_{\mathbb{L}^2}=0 \quad \forall \phi \in \mathbb{V}_h. \label{defi:l2-projection}
\end{align}
Find the $\{{\mathcal F}_{t_j};\, 0 \leq j \leq J\}$-adapted ${\mathbb V}_h$-valued process $\{ Y^j;\, 0 \leq j \leq J\}$ such that ${\mathbb P}$-almost surely
\begin{align}
&\bigl(Y^{j} - Y^{j-1}, \phi \bigr)_{{\mathbb L}^2} + k \Bigl[\bigl( \nabla Y^j , \nabla \phi \bigr)_{{\mathbb L}^2} + 
\bigl(f(Y^j, Y^{j-1}),\phi\bigr)_{{\mathbb L}^2}\Bigr]
 =  \Delta_j W \bigl( \sigma(Y^{j-1}), \phi\bigr)_{{\mathbb L}^2}\,\,\, \forall\, \phi \in {\mathbb V}_h, \label{eq:finite-element-discretization}
\end{align}
where 
 \begin{align}
  \Delta_j W:= W(t_j)- W(t_{j-1}) \sim \mathcal{N}(0,k),\,\,\,\text{and}\,\,\, f(y,z)= \big(|y|^2-1\big)\frac{y + z}{2}. \label{defi:brownian-increment-function-mixed}
 \end{align}
 Solvability for $k < 1$ is again immediate via Brouwer fixed point theorem. 

We are now in a position to state the main result of this article.
\begin{maintheorem*}
 Let the assumption \ref{A1} hold and $x\in {\mathbb W}^{2,2}_{{\rm per}}$. For every $\delta>0$, there exist constants $0< C_\delta < \infty$, independent of the discretized parameters
 $k, h >0$, and $k_0=k_0(T,x)>0$ such that for all $k\le k_0$ sufficiently small, there holds
 \begin{align}
  \sup_{1\le j\le J} \mathbb{E}\Big[ \|X_{t_j}-Y^j\|_{\mathbb{L}^2}^2\Big] + k\sum_{j=1}^J \mathbb{E}\Big[\|\grad(X_{t_j}-Y^j)\|_{\mathbb{L}^2}^2 \le C_\delta \big(k^{1-\delta} + h^2\big),\notag 
 \end{align}
 where $\big\{ X_t; t\in [0,T]\big\}$ solves \eqref{eq:Allen-Cahn} while $\big\{ Y^j; 0\le j\le J\big\}$ solves \eqref{eq:finite-element-discretization}.
\end{maintheorem*}
The proof is detailed in Sections \ref{sec:continuous-case}, \ref{sec:time-discrete} and \ref{sec:fully-discrete} and uses the semi-linear structure of $\mathscr{A}$ along with the weak monotonicity 
property \eqref{AD1}. We first consider a semi-discrete (in time) scheme \eqref{eq:time-discrete} of the problem \eqref{eq:Allen-Cahn} and derive the error estimate between the strong solution 
$X$ of \eqref{eq:Allen-Cahn} and the discretized solution $\big\{X^j;0\le j\le J\big\}$ of \eqref{eq:time-discrete}, see Theorem~\ref{thm:error-discrete in time}. Again, using 
uniform bounds for higher moments of the solutions $\{X^j\}$ and $\{Y^j\}$, we derive the error estimate of $X^j$ and $Y^j$ in strong norm, cf.~ Theorem \ref{thm:error-estime}. Putting things together then settles the main theorem. 
%%%%%%%%%%%%%%%%%%%%%%%%%%%%%%%%%%%%%%%%%%%%%%%%%%%%%%%%%%%%%%%%%%%%%%%%%%%%%%%%%%%%%%%%%%%%%%%%%%%%%%%%%%%%%%%
%%%%%%%%%%%%%%%%%%%%%%%%%%%%%%%%%%%%%%%%%%%%%%%%%%%%%%%%%%%%%%%%%%%%%%%%%%%%%%%%%%%%%%%%%%%%%%%%%%%%%%%%%%%%%%%%%%%%%%%%%
\section{Stochastic Allen-Cahn equation: The continuous case} \label{sec:continuous-case}
In this section, we derive uniform bounds of arbitrary moments for the strong solution $X\equiv \big\{X_t: t\in [0,T]\big\}$ of \eqref{eq:Allen-Cahn} and using these uniform bounds, we estimate the expectation of
the increment $\|X_t-X_s\|_{\mathbb{L}^2}^2$ in terms of $|t-s|$. 
\vspace{.1cm}

The following estimate may be shown by a standard Galerkin method which employs a (finite) sequence of
(${\mathbb W}^{1,2}_{\rm per}$-orthonormal) eigenfunctions $\{ w_j;\, 1 \leq j \leq N\}$ of the inversely compact, self-adjoint isomorphic operator
${\rm I} -\Delta: {\mathbb W}^{2,2}_{\rm per} \rightarrow {\mathbb L}^2_{\rm per}$, the use of It\^{o}'s formula to the functional 
$y\mapsto \mathcal{J}(y):=\frac{1}{2}\|\grad y\|_{\mathbb{L}^2}^2 + \psi(y)$, and the final
passage to the limit (see e.g.~\cite{gess1}),
\begin{align}
&\sup_{t \in [0,T]} {\mathbb E}\Bigl[ \mathcal{J}(X_t) + \int_0^t \Vert 
\Delta X_s - D\psi(X_s)\Vert^2_{{\mathbb L}^2}\, {\rm d}s\Bigr] \leq C \Bigl( 1+ {\mathbb E}\bigl[ \mathcal{J}(x)\bigr]\Bigr), \label{esti:w12-norm-strong-solun}
\end{align}
for which we require the improved regularity property $x \in {\mathbb W}^{1,2}_{\rm per}$. Thanks to the assumption \ref{A1}, we see that $\sigma$ satisfies 
 the following estimates: 
 \begin{itemize}
  \item [a)] There exists a constant $C>0$ such that 
  \begin{align}
   \Vert \nabla \sigma(\xi)\Vert^2_{{\mathbb L}^2} + \Bigl( D^2 \psi(\xi) \sigma(\xi), \sigma(\xi)\Bigr)_{{\mathbb L}^2} \leq
C\big(1 + \mathcal{J}(\xi)\big)\quad  \forall \xi \in {\mathbb W}^{1,2}_{\rm per}.\label{esti:assumption-1}
  \end{align}
  \item[b)] There exist constants $K_2, K_3, K_4 >0$ and $L_2, L_3, L_4 >0$ such that  for all $\xi\in {\mathbb W}^{2,2}_{\rm per} $
  \begin{align}\label{esti:assumption-2}
   \|\Delta \sigma(\xi)\|_{\mathbb{L}^2}^2 \le 
   \begin{cases}
    K_2 \|\Delta \xi\|_{\mathbb{L}^2}^2 + K_3 \|\grad \xi\|_{\mathbb{L}^2}^2 \|\Delta \xi\|_{\mathbb{L}^2}^2 + K_4 \|\grad \xi\|_{\mathbb{L}^2}^4 \quad \text{if}\,\,\,d=2 \\
    L_2\|\Delta \xi\|_{\mathbb{L}^2}^2 + L_3 \|\grad \Delta \xi\|_{\mathbb{L}^2}^{\frac{3}{2}} \|\grad \xi\|_{\mathbb{L}^2}^{\frac{5}{2}} + L_4 \|\grad \xi\|_{\mathbb{L}^2}^4
    \quad \text{if}\,\,\,d=3.
   \end{cases}
  \end{align}
 \end{itemize}
%%%%%%%%%%%%%%%%%%%%%%%%%%%%%%%%%%%%%%%%%%%%%%%%%%%%%%%%%%%%%%%%%%%%%%%%%%%%%%%%%%%%%%%%%%%%%%%%%%%%%%%%%%%%%%%%%%%%%%%%%%%%%%%%%%%%%%%%%%%%%%%%%%%%%%%%%%%%%%%

\begin{lem}\label{lem:higer-moment-functional-strong-solun}
Suppose that the assumption \ref{A1} holds and $p \in {\mathbb N}$. Then, there exists a constant $C \equiv C(\| x\|_{{\mathbb W}^{1,2}}, p,T) >0$ 
such that
\begin{itemize}
\item[(i)] \quad $\underset{t \in [0,T]}\sup {\mathbb E}\Bigl[ \Bigl(\mathcal{J}(X_t)\Bigr)^p\Bigr] + {\mathbb E}\Bigl[ \int_0^T  \Vert \nabla X_s\Vert^{2(p-1)}_{{\mathbb L}^2} \Vert \Delta X_s \Vert^2_{{\mathbb L}^2}\, {\rm d}s\Bigr] \leq C\, .$
\end{itemize}
Suppose in addition $x \in {\mathbb W}^{2,2}_{\rm per}$. Then, there exists a constant
$C \equiv C(\Vert x\Vert_{{\mathbb W}^{2,2}},T) >0$ such that
\begin{itemize}
\item[(ii)] \quad $\underset{t \in [0,T]}\sup {\mathbb E}\Bigl[  \Vert \Delta X_t\Vert^{2}_{{\mathbb L}^{2}}
\Bigr] +  {\mathbb E}\Bigl[ \int_0^T \Vert \nabla \Delta X_s \Vert^2_{{\mathbb L}^{2}}\, {\rm d}s\Bigr] \leq C\, .$
\end{itemize}
\end{lem}
%%%%%%%%%%%%%%%%%%%%%%%%%%%%%%%%%%%%%%%%%%%%%%%%%
\begin{proof} {\rm (i)} We proceed formally. Note that 
\begin{align*}
\mathcal{J}(\xi):=\frac{1}{2} \Vert \nabla \xi \Vert^2_{{\mathbb L}^2} + \psi(\xi) \qquad \text{and} \qquad -\mathscr{A}(\xi)\equiv \mathcal{J}^{\prime}(\xi)= - \Delta \xi + D\psi(\xi).
\end{align*}
We use It\^o's formula for  
$\xi \mapsto g(\xi) := \bigl(\mathcal{J}(\xi) \bigr)^{p}$. 
\begin{align}
D g(\xi) &=  -p \, \bigl( \mathcal{J}(\xi) \bigr)^{p-1} \mathscr{A}(\xi) \notag \\
D^2 g(\xi) &= p(p-1)  \, \bigl( \mathcal{J}(\xi) \bigr)^{p-2} \mathscr{A}(\xi) \otimes \mathscr{A}(\xi)
 + p \,  \bigl( \mathcal{J}(\xi) \bigr)^{p-1}  \bigl( -\Delta + D^2 \psi(\xi)\bigr)\, ,\nonumber
\end{align}
where $a \otimes b \cdot c = a (b,c)_{{\mathbb L}^2}$ for all $a,b,c \in {\mathbb L}^2$. By Cauchy-Schwartz inequality, and \eqref{esti:assumption-1}, we have
\begin{eqnarray*}\nonumber
&&{\mathbb E} \Bigl[ \bigl(\mathcal{J}(X_t)\bigr)^p - \bigl(\mathcal{J}(x)\bigr)^p 
+ p \int_0^t \bigl(\mathcal{J}(X_s)\bigr)^{p-1} \Vert \mathscr{A}(X_s)\Vert^2_{{\mathbb L}^2}\, {\rm d}s \Bigr]\\ \nonumber
&& = \frac{p}{2} {\mathbb E}\Biggl[ \int_0^t \Big\{ (p-1) \bigl( \mathcal{J}(X_s) \bigr)^{p-2}
\Bigl( \mathscr{A}(X_s) , \sigma(X_s)\Bigr)^2_{{\mathbb L}^2} \\ \nonumber 
&& \quad \qquad \qquad +  \bigl( \mathcal{J}(X_s)\bigr)^{p-1} \Bigl( [-\Delta + D^2 \psi(X_s)] \sigma(X_s), \sigma(X_s)\Bigr)_{{\mathbb L}^2}\Big\}\, {\rm d}s \Biggr] 
\\ \nonumber
&& \leq  \frac{p}{2}  {\mathbb E}  
\Big[ \int_0^t  (p-1)  \bigl( \mathcal{J}(X_s)\bigr)^{p-2} \Bigl( \mathscr{A}(X_s) , \sigma(X_s)\Bigr)^2_{{\mathbb L}^2}\,{\rm d}s\Big]
 + C(p)\int_0^t \mathbb{E}\Big[ \bigl( \mathcal{J}(X_s)\bigr)^p\Big] {\rm d}s + C.
\end{eqnarray*}
Since $\sigma$ is bounded, by using Young's inequality, we have 
\begin{align}
  \bigl(\mathcal{J}(X_s)\bigr)^{p-2} \Bigl( \mathscr{A}(X_s) , \sigma(X_s)\Bigr)^2_{{\mathbb L}^2} 
& \le C \bigl(\mathcal{J}(X_s)\bigr)^{p-2} \|\mathscr{A}(X_s)\|^2_{{\mathbb L}^2} \notag \\
& \le \theta \bigl( \mathcal{J}(X_s)\bigr)^{p-1} \|\mathscr{A}(X_s)\|^2_{{\mathbb L}^2} + C(\theta) \|\mathscr{A}(X_s)\|^2_{{\mathbb L}^2}. \notag 
\end{align}
We choose $\theta>0$ such that $p-\frac{p}{2}(p-1)\theta >0$. With this choice of $\theta$, by \eqref{esti:w12-norm-strong-solun},  we have for some constant $C_1=C_1(p)>0$, 
\begin{align}
 & \mathbb{E}\Big[\bigl( \mathcal{J}(X_t)\bigr)^p \Big] + C_1(p) 
 \mathbb{E}\Big[ \int_0^t \bigl( \mathcal{J}(X_s)\bigr)^{p-1} \Vert  \mathscr{A}(X_s)\Vert^2_{{\mathbb L}^2}\, {\rm d}s \Big] \notag \\
& \le \mathbb{E}\Big[\bigl( \mathcal{J}(x)\bigr)^p \Big] 
+ C(p)\int_0^t \mathbb{E}\Big[\bigl( \mathcal{J}(X_s)\bigr)^p \Big]\,{\rm d}s + C. \notag 
\end{align}
We use Gronwall's lemma to conclude that 
\begin{align}
 \sup_{t\in [0,T]} \mathbb{E}\Big[\bigl( \mathcal{J}(X_t)\bigr)^p \Big] 
 + \mathbb{E}\Big[ \int_0^T \bigl(\mathcal{J}(X_s)\bigr)^{p-1} \|\mathscr{A}(X_s)\|^2_{{\mathbb L}^2}\, {\rm d}s \Big] \le C. \label{esti:higer-moment-functional-strong-solun-1} 
\end{align}
In view of \eqref{esti:higher-moment-in-l2-strong-solun}, and \eqref{esti:higer-moment-functional-strong-solun-1} it follows that 
\begin{align}
 \sup_{t\in [0,T]}\mathbb{E}\Big[ \|X_t\|_{\mathbb{W}^{1,2}}^p\Big] \le C, \quad \forall p\ge 1. \label{esti:higher-moment-w12-strong-solun}
\end{align}
Note that, 
\begin{align}
\|\Delta X_s \|^2_{{\mathbb L}^2}\le \Vert \mathscr{A}(X_s)\Vert^2_{{\mathbb L}^2} +
\Vert X_s\Vert^6_{{\mathbb L}^6} + \Vert X_s\Vert_{{\mathbb L}^2}^2
 \leq \Vert \mathscr{A}(X_s)\Vert^2_{{\mathbb L}^2} + C \Vert X_s\Vert^6_{{\mathbb W}^{1,2}}
+ \Vert X_s\Vert_{{\mathbb W}^{1,2}}^2\, , \notag 
\end{align}
where we use the embedding ${\mathbb W}^{1,2} \hookrightarrow {\mathbb L}^6$ for $d \leq 3$. Thanks to 
\eqref{esti:higer-moment-functional-strong-solun-1} and \eqref{esti:higher-moment-w12-strong-solun}, together with Cauchy-Schwartz inequality and the above 
estimate, we see that 
\begin{align}
 &\mathbb{E}\Big[\int_0^T \bigl( \mathcal{J}(X_s)\bigr)^{p-1} \Vert  \Delta X_s\Vert^2_{{\mathbb L}^2}\, {\rm d}s \Big] \notag \\
& \le  C +  \mathbb{E}\Big[\int_0^T \bigl( \mathcal{J}(X_s)\bigr)^{p-1} \big( \Vert X_s\Vert^6_{{\mathbb W}^{1,2}} + \Vert X_s\Vert_{{\mathbb W}^{1,2}}^2\big) {\rm d}s \Big] \notag \\
& \le C + T \sup_{t\in [0,T]} \mathbb{E}\Big[\bigl( \mathcal{J}(X_t)\bigr)^p \Big]  + T 
\sup_{t\in [0,T]} \mathbb{E}\Big[\|X_t\|_{\mathbb{W}^{1,2}}^{6p} + \|X_t\|_{\mathbb{W}^{1,2}}^{2p} \Big] \le C.\label{esti:higer-moment-functional-strong-solun-2}
\end{align}
One can combine \eqref{esti:higer-moment-functional-strong-solun-1} and \eqref{esti:higer-moment-functional-strong-solun-2} to conclude the assertion.
\vspace{.2cm}

{\rm (ii)} Use It\^{o}'s formula for $\xi \mapsto g(\xi) = \frac{1}{2} \Vert \Delta \xi\Vert^{2}_{{\mathbb L}^{2}}$. We compute its derivatives.
\begin{eqnarray*}
Dg(\xi) = \Delta^2 \xi, \qquad  D^2g(\xi)= \Delta^2\, .
\end{eqnarray*}
Note that by integration by parts
\begin{align}
\Bigl( \Delta^2 X_s, -\Delta X_s + \vert X_s\vert^2 X_s\Bigr)_{{\mathbb L}^2}
 \ge \frac{1}{2} \Bigl[  \Vert \nabla \Delta X_s\Vert^2_{{\mathbb L}^2} -3 \Vert \vert X_s\vert^2 \nabla X_s\Vert^2_{{\mathbb L}^2}\Bigr]. \label{sac-1fg}
\end{align}
Because of ${\mathbb W}^{1,2} \hookrightarrow {\mathbb L}^6$ for $d \leq 3$ we 
estimate the last term through
\begin{align}
 \Vert \vert X_s\vert^2 \nabla X_s\Vert^2_{{\mathbb L}^2} & \leq  \Vert X_s\Vert^4_{{\mathbb L}^6} \Vert \nabla X_s\Vert^2_{{\mathbb L}^6} 
\leq C \Vert X_s\Vert^{4}_{{\mathbb W}^{1,2}} \big(\|\grad X_s\|_{\mathbb{L}^2}^2 + \|\Delta X_s\|_{\mathbb{L}^2}^2\big) \notag \\
& \le C\Vert X_s\Vert^{6}_{{\mathbb W}^{1,2}} + C \big( \|X_s\|_{\mathbb{L}^2}^2 + \|\grad X_s\|_{\mathbb{L}^2}^2\big)^2 \|\Delta X_s\|_{\mathbb{L}^2}^2 \notag \\
& \equiv C\Vert X_s\Vert^{6}_{{\mathbb W}^{1,2}} + \mathcal{M}. \label{esti:h2-norm-strong-1}
\end{align}
Note that $\|X_s\|_{\mathbb{L}^2}^2 \le C(1+ \psi(X_s))$, and therefore we see that 
\begin{align}
 \mathcal{M} \le C\Big( 1+ \big(\mathcal{J}(X_s)\big)^2\Big)\|\Delta X_s\|_{\mathbb{L}^2}^2. \label{esti:h2-norm-strong-2}
\end{align}

Inserting \eqref{esti:h2-norm-strong-1} and \eqref{esti:h2-norm-strong-2} in \eqref{sac-1fg}, we obtain 
\begin{align}
 & {\mathbb E}\Big[  \Vert \Delta X_t\Vert^2_{{\mathbb L}^2}\Big] +  \mathbb{E}\Big[\int_0^t  \| \nabla \Delta X_s \|^2_{{\mathbb L}^2}\,{\rm d}s \Big] \notag \\
 & \le \mathbb{E}\big[\Delta x\|_{\mathbb{L}^2}^2\big] + C\mathbb{E}\Big[\int_0^t \Vert X_s\Vert^{6}_{{\mathbb W}^{1,2}}\,{\rm d}s\Big]
 + C\mathbb{E}\Big[ \int_0^T \bigl(\mathcal{J}(X_s)\bigr)^2 \Vert  \Delta X_s\Vert^2_{{\mathbb L}^2}\,{\rm d}s\Big] \notag \\
 & \hspace{2cm} + C\int_{0}^t \mathbb{E}\big[\|\Delta X_s\|_{\mathbb{L}^2}^2\big]\,{\rm d}s
 + C\mathbb{E}\Big[\int_0^t \|\Delta \sigma(X_s)\|_{\mathbb{L}^2}^2\,{\rm d}s\Big]. \label{esti:w22-strong-solun-1}
\end{align}
Let $d=2$. Then by \eqref{esti:assumption-2}, we see that 
\begin{align}
 \mathbf{G}:& =\mathbb{E}\Big[\int_0^t \|\Delta \sigma(X_s)\|_{\mathbb{L}^2}^2\,{\rm d}s\Big] \notag \\
 & \le K_2 \int_{0}^t \mathbb{E}\big[\|\Delta X_s\|_{\mathbb{L}^2}^2\big]\,{\rm d}s + K_3 
 \mathbb{E}\Big[ \int_0^T \|\grad X_s\|_{\mathbb{L}^2}^2 \|\Delta X_s\|_{\mathbb{L}^2}^2\,{\rm d}s\Big] + K_4 \int_{0}^T \mathbb{E}\big[\|\grad X_s\|_{\mathbb{L}^2}^4\big]\,{\rm d}s.\notag 
\end{align}
One can combine the above estimate in \eqref{esti:w22-strong-solun-1} and use \eqref{esti:higher-moment-w12-strong-solun}, 
\eqref{esti:higer-moment-functional-strong-solun-2} along with Gronwall's lemma to conclude the assertion for $d\le 2$. 
\vspace{.1cm}

Let $d=3$. Then, thanks to \eqref{esti:assumption-2} and Cauchy-Schwartz inequality, we have 
\begin{align}
 \mathbf{G} &\le \frac{1}{2}\mathbb{E}\Big[\int_0^t  \| \nabla \Delta X_s \|^2_{{\mathbb L}^2}\,{\rm d}s \Big]
 + C \mathbb{E}\Big[\int_0^T  \| \nabla  X_s \|^{10}_{{\mathbb L}^2}\,{\rm d}s \Big] + L_4 \int_{0}^T \mathbb{E}\big[\|\grad X_s\|_{\mathbb{L}^2}^4\big]\,{\rm d}s \notag \\
 & \hspace{3cm}+ L_2 \int_{0}^t \mathbb{E}\big[\|\Delta X_s\|_{\mathbb{L}^2}^2\big]\,{\rm d}s.\label{esti:w22-strong-solun-2}
\end{align}
We combine \eqref{esti:w22-strong-solun-1} and \eqref{esti:w22-strong-solun-2} and use \eqref{esti:higher-moment-w12-strong-solun}, 
\eqref{esti:higer-moment-functional-strong-solun-2} along with Gronwall's lemma to conclude the assertion for $d=3$. 
 This completes the proof. 
\end{proof}
%%%%%%%%%%%%%%%%%%%%%%%%%%%%%%%%%%%%%%%%%%%%%%%%%%%%%%%%%%%%%%%%%%%%%%%%%%%%%%%%%%%%%%%%%%%%%%%%%%%%%%%%%%%%%%%%%%%%%%%%%%%%%%%%%%%%%
%%%%%%%%%%%%%%%%%%%%%%%%%%%%%%%%%%%%%%%%%%%%%%%%%%%%%%%%%%%%%%%%%%%%%%%%%%%%%%%%%%%%%%%%%%%%%%%%%%%%%%%%%%%%%%%%%%%%%%%%%%%%%

The following result is to bound the increments $X_t - X_s$ of the solutions of \eqref{eq:Allen-Cahn} in terms of $\vert t-s\vert^{\alpha}$ 
for some $\alpha > 0$; its proof uses Lemma~\ref{lem:higer-moment-functional-strong-solun} in particular.

\begin{lem}\label{lem:esti-time-difference-strong-solun}
Let the assumption \ref{A1} holds and $x \in {\mathbb W}^{2,2}_{\rm per}$ . Then for every $0 \leq s \leq t \leq T$, there exists a constant 
$C \equiv C(p,T)>0$ such that
\begin{enumerate}
\item[(i)] \quad ${\mathbb E}\bigl[ \Vert X_t - X_s\Vert^p_{{\mathbb L}^2}\bigr] \leq C \vert t-s\vert
\qquad (p\geq 2)$,
\item[(ii)] \quad ${\mathbb E}\bigl[ \Vert X_t - X_s\Vert^2_{{\mathbb W}^{1,2}}\bigr] \leq C \vert t-s\vert\, .$
\end{enumerate}
\end{lem}
%%%%%%%%%%%%%%%%%%%%%%%%%%%%%%%%%%%%%%%%%%%%%%%%%%%%%%%%%%%%%%%%%%%%%%%%%%%%%%%%%%%%%%%%%%%%%%%%%%%%%%%%%%%%%%%%%%%%%%%%%%%%%%%%%%%%%%%%
\begin{proof}
$\rm (i)$  Fix $s\ge 0$. An application of It\^{o}'s formula for $u\mapsto  \frac{1}{p}\|u-\beta\|_{\mathbb{L}^2}^p$ with $\beta=X_s(\cdot,\omega) \in \R $ to \eqref{eq:Allen-Cahn} 
 yields, after taking expectation 
\begin{align}
&{\mathbb E}\Bigl[ \frac{1}{p} \Vert X_t - X_s\Vert^p_{{\mathbb L}^2}+  \int_s^t
\Vert X_{\zeta} - X_s\Vert^{p-2}_{{\mathbb L}^2} \Bigl( \mathscr{A}(X_{\zeta}) - \mathscr{A}(X_{s}), X_{\zeta} - X_s\Bigr)_{{\mathbb L}^2}\,{\rm d}\zeta \Bigr] \notag \\
&\quad \leq {\mathbb E}\Bigl[ \int_s^t \Vert X_{\zeta} - X_s\Vert^{p-2}_{{\mathbb L}^2} \Bigl(-\mathscr{A}(X_{s}), X_{\zeta} - X_s \Bigr)_{{\mathbb L}^2} {\rm d}\zeta\Bigr]
 + C_p  {\mathbb E}\Bigl[\int_{s}^t \Vert X_{\zeta} - X_s\Vert^{p-2}_{{\mathbb L}^2} \Vert \sigma(X_{\zeta})\Vert^2_{{\mathbb L}^2}
\, {\rm d}\zeta\Bigr] \notag \\
& \quad \equiv A_1 + A_2. \notag 
\end{align}
We use the weak monotonicity property \eqref{AD1} to bound from below the second term on the left-hand side,
$$ \geq {\mathbb E}\Bigl[ \int_s^t \Big(\, \Vert  X_\zeta-X_s\Vert^{p-2}_{{\mathbb L}^2} \Vert \nabla(X_\zeta-X_s)\Vert^2_{{\mathbb L}^2}
- C\, \Vert X_{\zeta} - X_s\Vert^p_{{\mathbb L}^2}\Big){\rm d}{\zeta}\Bigr]\, .$$
 The integration by parts formula and Young's inequality reveal that 
\begin{align*}
 \Bigl(-\mathscr{A}(X_s), X_{\zeta} - X_s \Bigr)_{{\mathbb L}^2} &
 \le \|\grad X_s\|_{\mathbb{L}^2}\|\grad(X_\zeta -X_s)\|_{\mathbb{L}^2} + \|D\psi(X_s)\|_{\mathbb{L}^2}\|X_\zeta -X_s\|_{\mathbb{L}^2} \\
 & \le \Big(\|X_s\|_{\mathbb{L}^6}^3 + \|X_s\|_{\mathbb{W}^{1,2}}\Big)\|X_\zeta -X_s\|_{\mathbb{W}^{1,2}}
\end{align*}
Since ${\mathbb W}^{1,2} \hookrightarrow {\mathbb L}^6$ for $d \leq 3$, by using Young's inequality, we see that 
\begin{align}
 A_1 & \leq C\, {\mathbb E}\Bigl[ \int_s^t \Vert X_{\zeta} - X_s\Vert^{p-2}_{{\mathbb L}^2} \Bigl(\Vert X_s\Vert_{{\mathbb W}^{1,2}} + \Vert X_s\Vert_{{\mathbb L}^6}^3
\Bigr) \Vert  X_{\zeta} - X_s\Vert_{{\mathbb W}^{1,2}}\, {\rm d}\zeta\Bigr] \notag \\
&\leq {\mathbb E}\Bigl[ \frac{1}{2} \int_s^t\, \Vert  X_\zeta-X_s\Vert^{p-2}_{{\mathbb L}^2} 
\Bigl(\Vert  X_\zeta-X_s\Vert^2_{{\mathbb L}^2} + \Vert \nabla(X_\zeta-X_s)\Vert^2_{{\mathbb L}^2}\Bigr)
\, {\rm d}\zeta\Bigr] \notag  \\
& \hspace{1cm} + \frac{1}{2}\mathbb{E}\Big[ \int_s^t \Vert X_{\zeta} - X_s\Vert^{p-2}_{{\mathbb L}^2}\Big( \|X_s\|_{\mathbb{W}^{1,2}}^2 + \|X_s\|_{\mathbb{W}^{1,2}}^{6}
\Big) \,{\rm d}\zeta\Big] \notag \\
& \le  \frac{1}{2}\mathbb{E}\Big[ \int_s^t \Vert X_{\zeta} - X_s\Vert^{p-2}_{{\mathbb L}^2} \Vert \nabla(X_\zeta-X_s)\Vert^2_{{\mathbb L}^2}\,
{\rm d}\zeta \Big] + C_p \int_s^t \mathbb{E}\Big[\|X_\zeta -X_s\|_{\mathbb{L}^2}^p\,{\rm d}\zeta \Big] \notag  \\
& \hspace{2cm} + C |t - s|\, \sup_{\zeta \in [s,t]} {\mathbb E}\Bigl[ \Vert X_{\zeta}\Vert_{{\mathbb W}^{1,2}}^{p}  + \Vert X_{\zeta}\Vert_{{\mathbb W}^{1,2}}^{3p}\Bigr].\notag 
\end{align}
Again, thanks to \eqref{sac-1b} and Young's inequality, we see that 
\begin{align*}
 A_2 \le C_p \int_s^t \mathbb{E}\Big[\|X_\zeta -X_s\|_{\mathbb{L}^2}^p\,{\rm d}\zeta
 + |t-s| \sup_{\zeta \in [s,t]}\mathbb{E}\Big[\|X_\zeta\|_{\mathbb{L}^2}^p\Big].
\end{align*}
We combine all the above estimates and use \eqref{esti:higher-moment-in-l2-strong-solun} and  Lemma \ref{lem:higer-moment-functional-strong-solun}, $\rm (i)$ along with Gronwall's
inequality to get the result.
\vspace{.2cm}

$\rm (ii)$  We apply It\^{o}'s formula to the function $\frac{1}{2} |\nabla X_t - \beta|^2$ for any $\beta \in \R$ to \eqref{eq:Allen-Cahn}, and then use $\beta=\nabla X_s$ for fixed $0<s\le t$ and 
integrate with respect to spatial variable. Thanks to Young's inequality, and the boundedness of $\sigma^{\prime}$,
\begin{align}
\frac{1}{2} {\mathbb E}\Bigl[ \Vert \nabla(X_t - X_s)\Vert^2_{{\mathbb L}^2}\Bigr] &\le  \Bigl\vert \int_s^t
{\mathbb E}\Bigl[\bigl( -\Delta [X_{\zeta} - X_s],
\mathscr{A}(X_{\zeta})\bigr)_{{\mathbb L}^2}\Bigr]\, {\rm d}\zeta \Bigr\vert 
+ C\int_{s}^t {\mathbb E}\bigl[\Vert \nabla \sigma(X_{\zeta})\Vert^2_{{\mathbb L}^2}\bigr]\, {\rm d}\zeta \notag \\
&\leq C \vert t-s\vert \sup_{\zeta \in [s,t]} {\mathbb E}\Big[ \Vert \Delta X_\zeta\Vert_{{\mathbb L}^2}^2\Big] + C\int_s^t \mathbb{E}\Big[ \|\mathscr{A}(X_{\zeta})\|_{\mathbb{L}^2}^2 + 
\|\grad X_{\zeta}\|_{\mathbb{L}^2}^2\Big]\,{\rm d}\zeta.\notag 
\end{align}
Notice that $\|\mathscr{A}(X_{\zeta})\|_{\mathbb{L}^2}^2 \le \|\Delta X_{\zeta}\|_{\mathbb{L}^2}^2 + 
 C\big( \|X_\zeta\|_{\mathbb{L}^6}^6 + \|X_\zeta\|_{\mathbb{L}^2}^2\big)$. From the above estimate, and Lemma~\ref{lem:higer-moment-functional-strong-solun},
$\rm (ii)$, \eqref{esti:higher-moment-w12-strong-solun}, and the embedding ${\mathbb W}^{1,2} \hookrightarrow {\mathbb L}^6$
for $d \leq 3$, we conclude
\begin{align}
{\mathbb E}\Bigl[ \Vert \nabla(X_t - X_s)\Vert^2_{{\mathbb L}^2}\Bigr] \le  C \vert t-s\vert \sup_{\zeta \in [s,t]} {\mathbb E}\Big[ \Vert \Delta X_\zeta\Vert_{{\mathbb L}^2}^2 +
\Vert X_\zeta\Vert^6_{{\mathbb W}^{1,2}} + \Vert X_{\zeta}\Vert^2_{{\mathbb W}^{1,2}}\Bigr]\le  C |t-s|.\label{esti:grad}
\end{align}

One can use $\rm (i)$ of Lemma \ref{lem:esti-time-difference-strong-solun} for $p=2$, and \eqref{esti:grad} to arrive at $\rm (ii)$. This finishes the proof.
\end{proof}
%%%%%%%%%%%%%%%%%%%%%%%%%%%%%%%%%%%%%%%%%%%%%%%%%%%%%%%%%%%%%%%%%%%%%%%%%%%%%%%%%%%%%%%%%%%%%

\section{Semi-discrete scheme (in time) and its bound}\label{sec:time-discrete}
Let $0 = t_0 < t_1 < \ldots < t_J$ be an equi-distant partition of $[0,T]$ of size $k = \frac{T}{J}$. The structure preserving time discrete version of \eqref{eq:Allen-Cahn} defines 
an $\{{\mathcal F}_{t_j};\, 0 \leq j \leq J\}$-adapted ${\mathbb W}^{1,2}_{\rm per}$-valued process $\{ X^j;\, 0 \leq j \leq J\}$
such that ${\mathbb P}$-almost surely and for all $\phi \in {\mathbb W}^{1,2}_{{\rm per}}$
\begin{align} \label{eq:time-discrete}
\begin{cases}
\big(X^{j} - X^{j-1}, \phi\big)_{{\mathbb L}^2} + k \Big[\big( \nabla X^j , \nabla \phi \big)_{{\mathbb L}^2} + 
\big( f(X^j, X^{j-1}),\phi\big)_{{\mathbb L}^2}\Big]= 
\Delta_j W \big( \sigma(X^{j-1}), \phi\big)_{{\mathbb L}^2} \\
X^0=x\in {\mathbb L}^{2}_{\rm per},
\end{cases}
\end{align}
where $\Delta_j W$ and $f$ are defined in \eqref{defi:brownian-increment-function-mixed}. Solvability for $k < 1$ easily follows from a
coercivity property of the drift operator, and the Lipschitz continuity property \eqref{sac-1b} for the diffusion operator. Below, we denote again
\begin{align*}
 \mathcal{J}(X^j)= \frac{1}{2}\|\grad X^j\|_{\mathbb{L}^2}^2 + \psi(X^j). 
\end{align*}
%%%%%%%%%%%%%%%%%%%%%%%%%%%%%%%%%%%%%%%%%%%%%%%%%%%%%%%%%%%%%%%%%%%%%%%%%%%%%%%%%%%%%%%%%%%%%%
The proof of the following lemma evidences why $D\psi(X^j)$ is substituted by $f(X^j, X^{j-1})$ in \eqref{eq:time-discrete} to recover uniform bounds for arbitrary higher moments of 
$\mathcal{J}(X^j)$.
\begin{lem}\label{lem:moment-time-discrete}
Suppose $x \in {\mathbb W}^{1,2}_{\rm per}$, and that assumption \ref{A1} holds. For every $p = 2^r$, $r \in {\mathbb N}^*$, there exists a constant
$C \equiv C(p,T) > 0$ such that 
 \begin{align}
  &\max_{1\le j\le J} \mathbb{E}\Big[ \big|\mathcal{J}(X^j)\big|^p \Big] + \sum_{j=1}^J \mathbb{E}\Bigg[ \prod_{\ell=1}^r \big[ [\mathcal{J}(X^j)]^{2^{\ell-1}} + [\mathcal{J}(X^{j-1})]^{2^{\ell-1}}\big]\times 
  \Big( \|\grad (X^j-X^{j-1})\|_{\mathbb{L}^2}^2  \notag \\
  & \hspace{3.9cm }+ \big\| |X^j|^2- |X^{j-1}|^2\big\|_{\mathbb{L}^2}^2 + k \|-\Delta X^j + f(X^j, X^{j-1})\|_{\mathbb{L}^2}^2\Big)\Bigg]\le C. \notag 
 \end{align}
\end{lem}

%%%%%%%%%%%%%%%%%%%%%%%%%%%%%%%%%%%%%%%%%%%%%%%%%%%%%%%%%%%%%%%%%%%%%%%%%%%%%%%%%%%%%%%%%
\begin{proof} \textbf{1}. Consider \eqref{eq:time-discrete} for a fixed $\omega \in \Omega$ and choose
$\phi = -\Delta X^j(\omega) + f(X^j, X^{j-1})(\omega)$. Then one has $\mathbb{P}$-a.s.,
\begin{align}
 &\big(X^{j}-X^{j-1}, -\Delta X^j + f(X^j, X^{j-1})\big)_{\mathbb{L}^2} + k\big\|-\Delta X^j + f(X^j, X^{j-1})\big\|_{\mathbb{L}^2}^2 \notag \\
 &\quad = \Delta_j W\big( \grad \sigma(X^{j-1}), \grad X^j \big)_{{\mathbb L}^2} + \Delta_j W\big( \sigma(X^{j-1}), f(X^j, X^{j-1})\big)_{\mathbb{L}^2}=:\mathcal{A}_1 + \mathcal{A}_2.
 \label{eq:time-discrete-test-1}
\end{align}
By using the identity $(a-b)a=\frac{1}{2}\Big(|a|^2-|b|^2 + |a-b|^2\Big)\,\,\forall a,b \in \R$ along with integration by parts formula,
we calculate
\begin{align}
 \big(X^{j} &-X^{j-1}, -\Delta X^j + f(X^j, X^{j-1})\big)_{\mathbb{L}^2} \notag \\
 &= \big(\grad(X^j - X^{j-1}), \grad X^j\big)_{\mathbb{L}^2} + \frac{1}{2}\big( |X^j|^2-1, |X^j|^2-1 -(|X^{j-1}|^2-1)\big)_{\mathbb{L}^2} \notag \\
 &= \frac{1}{2}\Big( \|\grad X^j\|_{\mathbb{L}^2}^2 - \|\grad X^{j-1}\|_{\mathbb{L}^2}^2 + \|\grad (X^j-X^{j-1})\|_{\mathbb{L}^2}^2\Big) \notag \\
 & \hspace{2cm}+ \frac{1}{4} \Big( \||X^j|^2-1\|_{\mathbb{L}^2}^2 
 - \||X^{j-1}|^2-1\|_{\mathbb{L}^2}^2  + \big\| |X^j|^2- |X^{j-1}|^2\big\|_{\mathbb{L}^2}^2\Big)\notag \\
 & = \mathcal{J}(X^j)- \mathcal{J}(X^{j-1}) + \frac{1}{2} \|\grad (X^j-X^{j-1})\|_{\mathbb{L}^2}^2 + \frac{1}{4}\big\| |X^j|^2- |X^{j-1}|^2\big\|_{\mathbb{L}^2}^2. 
 \label{eq:time-discrete-test-2}
\end{align}
 Since $\sigma^{\prime}$ is bounded, we observe that
 \begin{align}
 \mathcal{A}_1 & \le \frac{1}{4} \|\grad(X^j-X^{j-1})\|_{\mathbb{L}^2}^2 + C \|\grad X^{j-1}\|_{\mathbb{L}^2}^2 |\Delta_j W|^2 
  + \big( \grad \sigma(X^{j-1}), \grad X^{j-1} \big)_{\mathbb{L}^2}\Delta_j W \notag \\
  & \le \frac{1}{4} \|\grad(X^j-X^{j-1})\|_{\mathbb{L}^2}^2 + C \mathcal{J}(X^{j-1}) |\Delta_j W|^2 
  + \big( \grad \sigma(X^{j-1}), \grad X^{j-1} \big)_{\mathbb{L}^2}\Delta_j W. \notag
 \end{align}
 We decompose $\mathcal{A}_2$ into the sum of two terms $\mathcal{A}_{2,1}$ and $\mathcal{A}_{2,2}$ where 
 \begin{align*}
  \begin{cases}
  \mathcal{A}_{2,1}= \big( \sigma(X^{j-1}), (|X^j|^2-|X^{j-1}|^2) \frac{X^j + X^{j-1}}{2}\big)_{\mathbb{L}^2} \Delta_j W \\
 \mathcal{A}_{2,2}= \big( \sigma(X^{j-1}), (|X^{j-1}|^2-1) \frac{X^j + X^{j-1}}{2}\big)_{\mathbb{L}^2} \Delta_j W .
  \end{cases}
 \end{align*}
 In view of Young's inequality and the boundedness of $\sigma$, we have 
\begin{align}
 &\mathcal{A}_{2,1} \le \frac{1}{8} \big\||X^j|^2- |X^{j-1}|^2\big\|_{\mathbb{L}^2}^2 +  C \big( \|X^j - X^{j-1}\|_{\mathbb{L}^2}^2 + \|X^{j-1}\|_{\mathbb{L}^2}^2\big) |\Delta_j W|^2, \notag \\
 & \mathcal{A}_{2,2} = \big( \sigma(X^{j-1}), (|X^{j-1}|^2-1) \frac{X^j - X^{j-1}}{2}\big)_{\mathbb{L}^2} \Delta_j W  + 
 \big( \sigma(X^{j-1}), (|X^{j-1}|^2-1)X^{j-1}\big)_{\mathbb{L}^2} \Delta_j W \notag \\
 &\qquad  \le \|X^j - X^{j-1}\|_{\mathbb{L}^2}^2 + \||X^{j-1}|^2-1\|_{\mathbb{L}^2}^2 |\Delta_j W|^2 + \big( \sigma(X^{j-1}), (|X^{j-1}|^2-1)X^{j-1}\big)_{\mathbb{L}^2} \Delta_j W \notag \\
 &\qquad  \le \|X^j - X^{j-1}\|_{\mathbb{L}^2}^2 + C \mathcal{J}(X^{j-1}) |\Delta_j W|^2 + \big( \sigma(X^{j-1}), (|X^{j-1}|^2-1)X^{j-1}\big)_{\mathbb{L}^2} \Delta_j W.\notag
\end{align}
Next we estimate $\|X^j- X^{j-1}\|_{\mathbb{L}^2}^2$ independently to bound $\mathcal{A}_{2,2}$. To do so, we choose as test function $\phi= (X^j- X^{j-1})(\omega)$ in \eqref{eq:time-discrete} and obtain 
\begin{align}
 &\|X^j - X^{j-1}\|_{\mathbb{L}^2}^2 + \frac{k}{2}\Big( \|\grad X^j\|_{\mathbb{L}^2}^2 - \|\grad X^{j-1}\|_{\mathbb{L}^2}^2 + \|\grad (X^j-X^{j-1})\|_{\mathbb{L}^2}^2\Big) \notag \\
 & \hspace{2cm}+ \frac{k}{2}\big( |X^j|^2-1, |X^j|^2- |X^{j-1}|^2\big)_{\mathbb{L}^2}
 = \big(\sigma(X^{j-1}), X^j- X^{j-1}\big)_{\mathbb{L}^2}\Delta_j W. \label{eq:l2-diff-time-discrete}
\end{align}
Note that 
\begin{equation}\label{esti:l2-diff-time-discrete-0}
 \begin{aligned}
 \frac{k}{2}\big( |X^j|^2-1, |X^j|^2- |X^{j-1}|^2\big)_{\mathbb{L}^2} &= k\Big( \psi(X^j)- \psi(X^{j-1}) + \frac{1}{4} \big\| |X^j|^2- |X^{j-1}|^2\big\|_{\mathbb{L}^2}^2\Big), \\
 \big(\sigma(X^{j-1}), X^j- X^{j-1}\big)_{\mathbb{L}^2}\Delta_j W  &\le \frac{1}{2} \|X^j - X^{j-1}\|_{\mathbb{L}^2}^2 + C\|X^{j-1}\|_{\mathbb{L}^2}^2|\Delta_j W|^2,
 \end{aligned}
\end{equation}
where in the last inequality we have used the Lipschitz continuous property of $\sigma$. We use \eqref{esti:l2-diff-time-discrete-0} in \eqref{eq:l2-diff-time-discrete} to get
\begin{align}
  &\|X^j - X^{j-1}\|_{\mathbb{L}^2}^2 \le C k \mathcal{J}(X^{j-1}) + C\|X^{j-1}\|_{\mathbb{L}^2}^2|\Delta_j W|^2. \label{esti:l2-diff-time-discrete}
\end{align}
Again, since $\mathscr{O}$ is a bounded domain, one has
\begin{align}
 \|X^{j-1}\|_{\mathbb{L}^2}^2 = \int_{\mathscr{O}} \big(|X^{j-1}|^2-1\big)\,{\rm d}x + |\mathscr{O}|
 \le C\big(1+ \frac{1}{4}\||X^{j-1}|^2-1\|_{\mathbb{L}^2}^2\big) \le C\big( 1+ \mathcal{J}(X^{j-1}\big),
 \label{esti:l2-norm-discrete-functional}
\end{align}
where $|\mathscr{O}|$ denotes the Lebesgue measure of $\mathscr{O}$. Combining the above estimates and then those for $\mathcal{A}_1$ and $\mathcal{A}_2$ in \eqref{eq:time-discrete-test-1}, and 
then \eqref{eq:time-discrete-test-2}, we obtain after taking expectation
\begin{align}
& \mathbb{E}\Big[ \mathcal{J}(X^j)-\mathcal{J}(X^{j-1})\Big] + \frac{1}{4}\mathbb{E}\Big[\|\grad (X^j-X^{j-1})\|_{\mathbb{L}^2}^2\Big] + \frac{1}{8} \mathbb{E}\Big[
\big\| |X^j|^2- |X^{j-1}|^2\big\|_{\mathbb{L}^2}^2\Big] \notag \\
& \hspace{2cm}+ k \mathbb{E}\Big[\|-\Delta X^j + f(X^j, X^{j-1})\|_{\mathbb{L}^2}^2\Big] \le Ck\Big( 1+ \mathbb{E}\big[ \mathcal{J}(X^{j-1})\big]\Big).\notag 
\end{align}
Summation over all time steps, and the discrete Gronwall's lemma then establish the assertion for $r=0$. 
\vspace{.1cm}

\textbf{2}.  In order to validate the assertion for $p=2^r, r\in \mathbb{N}^*$, we proceed inductively and illustrate the argument for $r=1$. Recall that we have from before
\begin{align}
 \mathcal{J}(X^j) &-\mathcal{J}(X^{j-1}) + \frac{1}{4} \|\grad(X^j-X^{j-1})\|_{\mathbb{L}^2}^2 + \frac{1}{8} \big\| |X^j|^2- |X^{j-1}|^2\big\|_{\mathbb{L}^2}^2  \notag \\
 & \hspace{3cm }+ 
 k \|-\Delta X^j + f(X^j, X^{j-1})\|_{\mathbb{L}^2}^2 \notag \\
 & \le C \mathcal{J}(X^{j-1})\Big( k(1+ |\Delta_j W|^2) + |\Delta_j W|^2\big(1+ |\Delta_j W|^2\big)\Big) + C|\Delta_j W|^2\big( 1+ |\Delta_j W|^2\big) \notag \\
 & \hspace{1.0cm}+ \big( \grad \sigma(X^{j-1}), \grad X^{j-1} \big)_{\mathbb{L}^2}\Delta_j W +  \big( \sigma(X^{j-1}), (|X^{j-1}|^2-1)X^{j-1}\big)_{\mathbb{L}^2} \Delta_j W.
 \label{inq:l2-diff-time-discrete-1}
\end{align}
To prove the assertion for $r=1$, one needs to multiply \eqref{inq:l2-diff-time-discrete-1} by some quantity to produce a term like
$\mathcal{J}^2(X^j)-\mathcal{J}^2(X^{j-1}) + \alpha \big| \mathcal{J}(X^j)-\mathcal{J}(X^{j-1})\big|^2$ with $\alpha>0$ on the left hand side of the inequality in order to absorb 
related terms coming from the right-hand side of the inequality before discrete 
Gronwall's lemma. Therefore, we multiply \eqref{inq:l2-diff-time-discrete-1} with $\mathcal{J}(X^j)+ \frac{1}{2} \mathcal{J}(X^{j-1})$ to get by binomial formula
\begin{align}
 \frac{3}{4} &\Big( \mathcal{J}^2(X^j)-\mathcal{J}^2(X^{j-1})\Big) + \frac{1}{4} \big| \mathcal{J}(X^j)-\mathcal{J}(X^{j-1})\big|^2  + \frac{1}{2}\big( \mathcal{J}(X^j)+ \mathcal{J}(X^{j-1})\big)
 \times \notag \\
 & \hspace{2cm}   \Big\{ \frac{1}{4} \|\grad(X^j-X^{j-1})\|_{\mathbb{L}^2}^2 + \frac{1}{8} \big\| |X^j|^2- |X^{j-1}|^2\big\|_{\mathbb{L}^2}^2 + 
 k \|-\Delta X^j + f(X^j, X^{j-1})\|_{\mathbb{L}^2}^2\Big\} \notag \\
 & \le C \mathcal{J}(X^{j-1}) \big( \mathcal{J}(X^j)+ \frac{1}{2} \mathcal{J}(X^{j-1})\big)\Big\{ k(1+ |\Delta_j W|^2) + |\Delta_j W|^2\big(1+ |\Delta_j W|^2\big)\Big\} \notag \\
 & \quad  + C\big(\mathcal{J}(X^j)+ \frac{1}{2} \mathcal{J}(X^{j-1})\big)|\Delta_j W|^2\big( 1+ |\Delta_j W|^2\big)  \notag \\
 & \qquad + \big(\mathcal{J}(X^j)+ \frac{1}{2} \mathcal{J}(X^{j-1})\big)\big( \sigma(X^{j-1}), (|X^{j-1}|^2-1)X^{j-1}\big)_{\mathbb{L}^2} \Delta_j W  \notag \\
 & \quad \qquad + \big(\mathcal{J}(X^j)+ \frac{1}{2} \mathcal{J}(X^{j-1})\big)\big( \grad \sigma(X^{j-1}), \grad X^{j-1} \big)_{\mathbb{L}^2}\Delta_j W
 := \mathcal{A}_3 + \mathcal{A}_4 + \mathcal{A}_5 + \mathcal{A}_6. \label{inq:higher-moment-time-discrete-0}
\end{align}
By Young's inequality, we have $(\theta_1, \theta_2>0)$
\begin{align}
 \mathcal{A}_3 & \le \theta_1 \big| \mathcal{J}(X^j)-\mathcal{J}(X^{j-1})\big|^2 + C(\theta_1) \mathcal{J}^2(X^{j-1})\Big\{ k(1+ |\Delta_j W|^2)
 + |\Delta_j W|^2\big(1+ |\Delta_j W|^2\big)\Big\}^2 \notag \\
 & \hspace{2cm} + C\mathcal{J}^2(X^{j-1}) \Big\{ k(1+ |\Delta_j W|^2) + |\Delta_j W|^2\big(1+ |\Delta_j W|^2\big)\Big\}, \notag \\
 \mathcal{A}_4 & \le \theta_2 \big| \mathcal{J}(X^j)-\mathcal{J}(X^{j-1})\big|^2  + C(\theta_2)|\Delta_j W|^4(1+ |\Delta_j W|^4) + C \mathcal{J}^2(X^{j-1})|\Delta_j W|^4 \notag \\
 & \hspace{4cm} + C(1+ |\Delta_j W|^4). \notag 
\end{align}
We can decompose $\mathcal{A}_6$ as 
\begin{align}
 \mathcal{A}_6 & = \big(\mathcal{J}(X^j)- \mathcal{J}(X^{j-1})\big)\big( \grad \sigma(X^{j-1}), \grad X^{j-1} \big)_{\mathbb{L}^2}\Delta_j W \notag \\
 & \hspace{1cm}+ \frac{3}{2} 
  \mathcal{J}(X^{j-1})\big)\big( \grad \sigma(X^{j-1}), \grad X^{j-1} \big)_{\mathbb{L}^2}\Delta_j W:= \mathcal{A}_{6,1}+ \mathcal{A}_{6,2}. \notag 
\end{align}
Note that $\mathbb{E}\big[\mathcal{A}_{6,2}\big]=0$. By using Young's inequality and the boundedness of $\sigma^{\prime}$, we estimate $\mathcal{A}_{6,1}$, 
\begin{align}
 \mathcal{A}_{6,1}& \le \theta_3 \big| \mathcal{J}(X^j)-\mathcal{J}(X^{j-1})\big|^2 + C(\theta_3) \|\grad X^{j-1}\|_{\mathbb{L}^2}^4 |\Delta_j W|^2 \notag \\
 & \le \theta_3 \big| \mathcal{J}(X^j)-\mathcal{J}(X^{j-1})\big|^2 + C(\theta_3) |\Delta_j W|^2 \big( 1+ \mathcal{J}^2(X^{j-1})\big). \notag 
\end{align}
Again, $\mathcal{A}_5$ can be written as $\mathcal{A}_{5,1} + \mathcal{A}_{5,2}$ with $\mathbb{E}\big[\mathcal{A}_{5,2}]=0$, where $$\mathcal{A}_{5,1}= 
\big(\mathcal{J}(X^j)-\mathcal{J}(X^{j-1})\big)\big( \sigma(X^{j-1}), (|X^{j-1}|^2-1)X^{j-1}\big)_{\mathbb{L}^2} \Delta_j W.$$
Thanks to Young's inequality, the boundedness of $\sigma$ and \eqref{esti:l2-norm-discrete-functional} we get for $\theta_4>0$
\begin{align}
 \mathcal{A}_{5,1}& \le \theta_4 \big| \mathcal{J}(X^j)-\mathcal{J}(X^{j-1})\big|^2 + C(\theta_4) \big\||X^{j-1}|^2 -1\big\|_{\mathbb{L}^2}^2 \|X^{j-1}\|_{\mathbb{L}^2}^2 |\Delta_j W|^2 \notag \\
 & \le \theta_4 \big| \mathcal{J}(X^j)-\mathcal{J}(X^{j-1})\big|^2 + C(\theta_4) \big( \mathcal{J}^2(X^{j-1})+ \|X^{j-1}\|_{\mathbb{L}^2}^4\big) |\Delta_j W|^2 \notag \\
 & \le \theta_4 \big| \mathcal{J}(X^j)-\mathcal{J}(X^{j-1})\big|^2 + C(\theta_4) \big(1+ \mathcal{J}^2(X^{j-1})\big) |\Delta_j W|^2. \notag 
\end{align}
We combine all the above estimates in \eqref{inq:higher-moment-time-discrete-0}, and choose $\theta_1, \cdots,\theta_4 >0$ with $\sum_{i=1}^4 \theta_i < \frac{1}{4}$ to have, after taking expectation 
\begin{align}
 & \mathbb{E}\Big[ \mathcal{J}^2(X^j)-\mathcal{J}^2(X^{j-1}) + C_1 \big| \mathcal{J}(X^j)-\mathcal{J}(X^{j-1})\big|^2 \Big] + C_2 \mathbb{E}\Big[\big( \mathcal{J}(X^j)+ \mathcal{J}(X^{j-1})\big) \notag \\
 & \hspace{2cm} \times \Big\{\|\grad(X^j-X^{j-1})\|_{\mathbb{L}^2}^2 +  \big\| |X^j|^2- |X^{j-1}|^2\big\|_{\mathbb{L}^2}^2 + k \|-\Delta X^j + f(X^j, X^{j-1})\|_{\mathbb{L}^2}^2\Big\}\Big] \notag \\
 & \quad \le C_3(1+k) + C_4 k \mathbb{E}\Big[ \mathcal{J}^2(X^{j-1})\Big].\label{inq:higher-moment-time-discrete-1}
\end{align}
Summation over all time steps $0\le j\le J$ in \eqref{inq:higher-moment-time-discrete-1}, together with the discrete 
Gronwall's lemma then validates the assertion of the theorem for $r=1$. This completes the proof. 
\end{proof}
%%%%%%%%%%%%%%%%%%%%%%%%%%%%%%%%%%%%%%%%%%%%%%%%%%%%%%%%%%%%%%%%%%%%%%%%%%%%%%%%%%%%%%%%%%%%%%%%%%%%%%%%%%%%%%%%%%%%%%%%%%%
%%%%%%%%%%%%%%%%%%%%%%%%%%%%%%%%%%%%%%%%%%%%%%%%%%%%%%%%%%%%%%%%%%%%%%%%%%%%%%%%%%%%%%%%%%%%%%%%%%%%%%%%%%%%%%%%%%%%%%%%%%%%%%%%%%%%%%%%%%%%%%%%%%%%%%%%
We employ the bounds for arbitrary moments of $X$ in the strong norms in Lemma~\ref{lem:higer-moment-functional-strong-solun}, $\rm(i)$, and 
a weak monotonicity argument to prove the following error estimate for the solution $\{ X^j;\, 0 \leq j \leq J\}$ of \eqref{eq:time-discrete}.
%%%%%%%%%%%%%%%%%%%%%%%%%%%%%%%%%%%%%%%%%%%%%%%%%%%%%%%%%%%%%%%%%%%%%%%%%%%%%%%%%%

\begin{thm}\label{thm:error-discrete in time}
Assume that $x \in {\mathbb W}^{2,2}_{\rm per}$, and the assumption \ref{A1} holds true. Then, for every $\delta >0$, there exist constants
$0 \leq C_{\delta} < \infty$ and $k_1=k_1(x,T)>0$ such that for all $k\le k_1$ sufficiently small
\begin{align}
\sup_{0 \leq j \leq J} {\mathbb E}\bigl[ \Vert X_{t_j} - X^j\Vert^2_{{\mathbb L}^2}\bigr] + k \sum_{j=0}^J 
{\mathbb E}\bigl[\Vert \grad(X_{t_j} - X^j) \Vert^2_{{\mathbb L}^{2}}\bigr] \leq C_{\delta} k^{1-\delta}, \notag 
\end{align}
where $\{X_t;\, t \in [0,T] \}$ solves \eqref{eq:variational-Allen-Cahn} while $\{ X^j;\, 0 \leq j \leq J\}$ solves \eqref{eq:time-discrete}.
\end{thm}
The parameter $\delta>0$ which appears in Theorem~\ref{thm:error-discrete in time} is due to the non-Lipschitz drift in the problem and is caused by the estimate 
\eqref{esti:III} below.
%%%%%%%%%%%%%%%%%%%%%%%%%%%%%%%%%%%%%%%%%%%%%%%%%%%%%%%%%%%%%%%%%%%%%%%%%%%%%%%%%5
\begin{proof}
Consider \eqref{eq:variational-Allen-Cahn} for the time interval $[t_{j-1},t_j]$, and denote $e^j := X_{t_j} - X^j$. There holds ${\mathbb P}$-a.s.~for all $\phi \in {\mathbb W}^{1,2}_{\rm per}$
\begin{equation}\label{eq:error-time-discrete}
 \begin{aligned}
  & \big(e^j- e^{j-1}, \phi\big)_{{\mathbb L}^2} + \int_{t_{j-1}}^{t_j} \Big(\bigl(\nabla [X_{t_j}-X^{j}],
\nabla \phi \bigr)_{{\mathbb L}^2} + \bigl( D\psi(X_{t_j}) - D\psi(X^j), \phi\bigr)_{{\mathbb L}^2}\Big) {\rm d}s \\
& = - \int_{t_{j-1}}^{t_j}\bigl( \nabla [X_s - X_{t_j}], \nabla \phi\bigr)_{{\mathbb L}^2}\,{\rm d}s - \int_{t_{j-1}}^{t_j} \bigl( D \psi(X_{s}) - D\psi(X_{t_j}), \phi\bigr)_{{\mathbb L}^2} {\rm d}s \\
& \quad - \frac{1}{2}\int_{t_{j-1}}^{t_j} \big((|X^j|^2-1)(X^j-X^{j-1}), \phi\big)_{\mathbb{L}^2}\,{\rm d}s
 + \int_{t_{j-1}}^{t_j} \bigl( \sigma(X_s) - \sigma(X_{t_{j-1}}), \phi\bigr)_{{\mathbb L}^2}\, {\rm d}W_s \\
& \hspace{3cm}- \int_{t_{j-1}}^{t_j} \bigl( \sigma(X_{t_{j-1}}) - \sigma(X^{j-1}), \phi\bigr)_{{\mathbb L}^2}\, {\rm d}W_s \\
& =: I_j + II_j + III_j + IV_j + V_j.
 \end{aligned}
\end{equation}
The third term on the right-hand side attributes to the use of $f(X^j, X^{j-1})$ instead of $D\psi(X^j)$ in \eqref{eq:time-discrete}. 
Choose $\phi = e^j(\omega)$, and apply expectation. By the weak monotonicity property \eqref{AD1} of the drift, the left-hand side of \eqref{eq:error-time-discrete} is then bounded from below by
\begin{eqnarray*}
 \frac{1}{2} {\mathbb E}\Bigl[ \Vert e^j\Vert^2_{{\mathbb L}^2} - \Vert e^{j-1}\Vert^2_{{\mathbb L}^2} + 
\Vert e^{j} - e^{j-1}\Vert^2_{{\mathbb L}^2} + 2k\, \bigl(\Vert \nabla e^j\Vert^2_{{\mathbb L}^2} - \Vert e^j\Vert^2_{{\mathbb L}^2}\bigr)\Bigr]\, .
\end{eqnarray*}
Because of Young's inequality and Lemma~\ref{lem:esti-time-difference-strong-solun}, $\rm (ii)$  we conclude
$$ {\mathbb E}\bigl[ I_j\bigr] \leq Ck^2 + \frac{k}{8} \mathbb{E}\Big[ \Vert \nabla e^j\Vert^2_{{\mathbb L}^2}\Big] \,.$$
Next we bound ${\mathbb E}[II_j]$. For this purpose, we use  the embedding ${\mathbb W}^{1,2} \hookrightarrow {\mathbb L}^6$ for $d \leq 3$, the algebraic identity
$a^3 - b^3 = \frac{1}{2}(a-b) \bigl((a+b)^2 + a^2 + b^2\bigr)$,  and 
Young's and H\"older's inequalities in combination with Lemma \ref{lem:esti-time-difference-strong-solun} to estimate ($\delta >0$)
\begin{align}
{\mathbb E}[II_j] & \le \frac{1}{2} \int_{t_{j-1}}^{t_j}{\mathbb E}\Bigl[ \Vert X_s - X_{t_j}\Vert_{{\mathbb L}^2}
\Vert (X_s + X_{t_j})^2 + X_s^2 + X_{t_j}^2\Vert_{{\mathbb L}^3} \Vert e^j\Vert_{{\mathbb L}^6} \Bigr] \, {\rm d}s \notag \\
&  \qquad \qquad + \int_{t_{j-1}}^{t_j} \mathbb{E}\Big[\|X_s-X_{t_j}\|_{\mathbb{L}^2}\|e^j\|_{\mathbb{L}^2}\Big]\, {\rm d}s \notag \\
&\leq Ck\,  \sup_{s \in [t_{j-1},t_j]} {\mathbb E}\Bigl[\Vert X_s - X_{t_j}\Vert_{{\mathbb L}^2}^2 
\Bigl( \Vert X_{t_j}\Vert_{{\mathbb L}^6}^4 + \Vert X_{s}\Vert_{{\mathbb L}^6}^4\Bigr) \Bigr]
+ \frac{k}{8} {\mathbb E}\Bigl[\Vert e^j\Vert^2_{{\mathbb W}^{1,2}}\Bigr] + Ck^2 \notag \\
&\leq Ck \, \sup_{s \in [t_{j-1},t_j]} \Bigl(  {\mathbb E}\Bigl[ \Vert X_s - X_{t_j}\Vert^{2(1+\delta)}_{{\mathbb L}^2}\Bigr] 
 \Bigr)^{\frac{1}{1+\delta}}\Bigl(  {\mathbb E}\Bigl[ 
\bigl( \Vert X_{t_j}\Vert_{{\mathbb W}^{1,2}}^4 + \Vert X_{s}\Vert_{{\mathbb W}^{1,2}}^4\bigr)^{\frac{1+ \delta}{\delta}}
\Bigr]\Bigr)^{\frac{\delta}{1+\delta}} \notag \\
&  + \frac{k}{8} {\mathbb E}\Bigl[\Vert \grad e^j\Vert^2_{{\mathbb L}^2}\Bigr] + Ck \mathbb{E}\big[\|e^j\|_{\mathbb{L}^2}^2\big] + Ck^2. \label{esti:III}
\end{align}
The leading factor is bounded by $Ck^{\frac{1}{1+\delta}}$ by Lemma~\ref{lem:esti-time-difference-strong-solun}, $\rm(i)$, while the second factor may
be bounded by $C_{\delta}$ due to \eqref{esti:higher-moment-w12-strong-solun}. Thus we have 
\begin{align}
 {\mathbb E}[II_j] \le C_\delta k^{\frac{2+\delta}{1+\delta}} + Ck^2 + \frac{k}{8} {\mathbb E}\Bigl[\Vert \grad e^j\Vert^2_{{\mathbb L}^{2}}\Bigr] 
 + Ck \mathbb{E}\big[\|e^j\|_{\mathbb{L}^2}^2\big]. \notag 
\end{align}
It is immediate to validate $$\bigl\vert{\mathbb E}[IV_j] \bigr\vert + \bigl\vert {\mathbb E}[V_j] \bigr\vert \leq Ck^2 +
\frac{1}{8} \Vert e^j - e^{j-1}\Vert^2_{{\mathbb L}^2} + Ck\, \Vert e^{j-1}\Vert^2_{{\mathbb L}^2}$$ by adding and subtracting $e^{j-1}$ in
the second argument and proceeding as before, and It\^{o}'s isometry in combination with \eqref{sac-1b} and Lemma~\ref{lem:esti-time-difference-strong-solun}, $\rm (i)$. 
Next we focus on the term $III_j$. In view of generalized H\"{o}lder's inequality, and the embedding ${\mathbb W}^{1,2} \hookrightarrow {\mathbb L}^6$ for $d \leq 3$,
\begin{align}
 \mathbb{E}\big[ III_j\big] & \le \frac{1}{2} \mathbb{E}\Big[ \int_{t_{j-1}}^{t_j} \|e^j\|_{\mathbb{L}^6} \|X^j-X^{j-1}\|_{\mathbb{L}^2} \||X^j|^2-1\|_{\mathbb{L}^3}\,{\rm d}s \Big] \notag \\
 & \le \frac{k}{8} \mathbb{E}\big[\|e^j\|_{\mathbb{L}^6}^2\big] + Ck \mathbb{E}\Big[\|X^j-X^{j-1}\|_{\mathbb{L}^{2}}^2\||X^j|^2-1\|_{\mathbb{L}^3}^2\Big] \notag \\
 & \le  \frac{k}{8} \mathbb{E}\big[\|e^j\|_{\mathbb{W}^{1,2}}^2\big] + Ck\mathbb{E}\Big[ \|X^j-X^{j-1}\|_{\mathbb{L}^{2}}^2 \big( 1+ \|X^j\|_{\mathbb{W}^{1,2}}^4\big)\Big]\notag \\
 & \quad \equiv \frac{k}{8} \mathbb{E}\big[\|e^j\|_{\mathbb{W}^{1,2}}^2\big] + III_{j,1}. \notag 
\end{align}
In view of Lemma \ref{lem:moment-time-discrete}, we see that
\begin{align}
\sup_j \mathbb{E}\big[\|X^j\|_{\mathbb{W}^{1,2}}^p\big] \le C \quad \text{for any}\,\,\, p\ge 2, \label{esti:higher-moment-w12-time-discrete}
\end{align}
We use \eqref{esti:l2-diff-time-discrete}, Lemma \ref{lem:moment-time-discrete} and \eqref{esti:higher-moment-w12-time-discrete} to estimate $III_{j,1}$,
\begin{align*}
 III_{j,1} & \le Ck\mathbb{E}\Big[ \big(k \mathcal{J}(X^{j-1}) + \|X^{j-1}\|_{\mathbb{L}^2}^2|\Delta_j W|^2\big)  \big( 1+ \|X^j\|_{\mathbb{W}^{1,2}}^4\big)\Big]\notag \\
 & \le Ck^2 \mathbb{E}\Big[ \mathcal{J}(X^{j-1}) + \big(\mathcal{J}(X^{j-1})\big)^2 + \|X^j\|_{\mathbb{W}^{1,2}}^8 + \|X^{j-1}\|_{\mathbb{L}^2}^2\Big] \notag \\
 & \qquad \quad + C \mathbb{E}\Big[ k\|X^{j-1}\|_{\mathbb{L}^2}^2 |\Delta_j W|^2 \|X^j\|_{\mathbb{W}^{1,2}}^4\Big] \notag \\
 & \le Ck^2 + Ck^2\mathbb{E}\Big[ \|X^j\|_{\mathbb{W}^{1,2}}^8\Big] + C \mathbb{E}\Big[ \|X^{j-1}\|_{\mathbb{L}^2}^4 |\Delta_j W|^4\Big] \le Ck^2, \notag 
\end{align*}
and therefore we obtain 
\begin{align}
 \mathbb{E}\big[ III_j\big] \le \frac{k}{8} {\mathbb E}\Bigl[\Vert \grad e^j\Vert^2_{{\mathbb L}^2}\Bigr] + Ck  \mathbb{E}\big[\|e^j\|_{\mathbb{L}^2}^2\big] + Ck^2. \notag 
\end{align}
We combine all the above estimates to have 
\begin{align}
{\mathbb E}\Bigl[ \Vert e^j\Vert^2_{{\mathbb L}^2} - \Vert e^{j-1}\Vert^2_{{\mathbb L}^2} 
+ k\Vert \nabla e^j\Vert^2_{{\mathbb L}^2}\Bigr] \le Ck^2 + C_{\delta}k^{\frac{2+\delta}{1+\delta}}
+ Ck\Big( \mathbb{E}\Big[ \|e^j\|_{\mathbb{L}^2}^2 + \|e^{j-1}\|_{\mathbb{L}^2}^2\Big]\Big).\label{esti:error-1}
\end{align}
Summation over all time steps $0 \leq j \leq J$ in \eqref{esti:error-1}, together with 
the discrete (implicit form) Gronwall's lemma then validates the assertion of the theorem.
\end{proof}
%%%%%%%%%%%%%%%%%%%%%%%%%%%%%%%%%%%%%%%%%%%%%%%%%%%%%%%%%%%%%%%%%%%%%%%%%%%%%%%%%%%%%%%%%%%%%%%%%%%%%%%%%%%%%%%%%%%%%%%%%%%%%%%%%%%%%%%%%%%%%%%%%%%%%%%%%%%%
\section{Space-time discretization and strong error estimate} \label{sec:fully-discrete}
In this section, we first derive the uniform moment estimate for the discretized solution $\big\{Y^j:0\le j\le J\big\}$ of the structure preserving finite element based fully discrete scheme 
\eqref{eq:finite-element-discretization}. Then by using these uniform bounds along with Lemma \ref{lem:moment-time-discrete} we bound the error $E^j := X^j - Y^j$, where $\{ X^j;\, 0 \leq j \leq J\}$ 
solves \eqref{eq:time-discrete}. 
\vspace{.2cm}

 We define the discrete Laplacian $\Delta_h: \mathbb{V}_h\rightarrow \mathbb{V}_h$ by the variational identity 
\begin{align}
 -\big( \Delta_h \phi_h, \psi_h\big)_{\mathbb{L}^2}= \big( \nabla \phi_h, \nabla \psi_h\big)_{\mathbb{L}^2} \quad \forall \phi_h, \psi_h \in \mathbb{V}_h. \notag 
\end{align}
One can use the test function $\phi=-\Delta_h Y^j + \mathscr{P}_{\mathbb{L}^2}f(Y^j, Y^{j-1}) \in \mathbb{V}_h$ in \eqref{eq:finite-element-discretization} and proceed as in the proof of 
Lemma \ref{lem:moment-time-discrete} along with \eqref{defi:l2-projection}, the $\mathbb{W}^{1,2}$ and $\mathbb{L}^q\, (1\le q\le \infty)$-stabilities of the projection
operator $\mathscr{P}_{\mathbb{L}^2}$ (cf.~ \cite{douglas-dupont-wahlbin}) to arrive at the following uniform moment estimates for $\big\{ Y^j; 0\le j\le J\big\}$. 

\begin{lem}\label{lem:moment-space-time-discrete}
 For every $p = 2^r,\, r \in {\mathbb N}^*$, there exists a constant $C \equiv C(p,T) > 0$ such that
 \begin{align}
  &\max_{1\le j\le J} \mathbb{E}\Big[ \big|\mathcal{J}(Y^j)\big|^p \Big] + \sum_{j=1}^J \mathbb{E}\Bigg[ \prod_{\ell=1}^r \big[ [\mathcal{J}(Y^j)]^{2^{\ell-1}} + [\mathcal{J}(Y^{j-1})]^{2^{\ell-1}}\big]\times 
  \Big( \|\grad (Y^j-Y^{j-1})\|_{\mathbb{L}^2}^2  \notag \\
  & \hspace{3.9cm }+ \big\| |Y^j|^2- |Y^{j-1}|^2\big\|_{\mathbb{L}^2}^2 + k \big\|-\Delta_h Y^j + \mathscr{P}_{\mathbb{L}^2} f(Y^j, Y^{j-1})\big\|_{\mathbb{L}^2}^2\Big)\Bigg]\le C, \notag 
 \end{align}
 provided $\mathbb{E}\big[|\mathcal{J}(Y^0)|^p\big] \le C$. 
\end{lem}
In view of Lemma \ref{lem:moment-space-time-discrete}, it follows that 
\begin{align}
 \sup_{0\le j\le J} \mathbb{E}\Big[ \|Y^j\|_{\mathbb{W}^{1,2}}^p\Big] \le C, \quad \forall p\ge 2. \label{esti:higher-moment-w12-space-time-discrete}
\end{align}

%%%%%%%%%%%%%%%%%%%%%%%%%%%%%%%%%%%%%%%%%%%%%%%%%%%%%%%%%%%
We have the following theorem regarding the error $E^j$ in strong norm. 

\begin{thm}\label{thm:error-estime}
Assume that $x \in {\mathbb W}^{2,2}_{\rm per}$. Then, under the assumption \ref{A1}, there exist constants $C>0$, independent of the discretization parameters $h,k >0$ and 
$k_2\equiv k_2(T, x)>0$ such that for all $k\le k_2$  sufficiently small, there holds
\begin{align*}
\sup_{0 \leq j \leq J} {\mathbb E}\Bigl[ \| X^j - Y^j\|^2_{{\mathbb L}^2}\Bigr] + k \sum_{j=0}^J 
{\mathbb E}\Bigl[\|\grad( X^j - Y^j) \|^2_{{\mathbb L}^{2}}\Bigr] \leq C\big(k+h^2\big), 
\end{align*}
where $\{ X^j;\, 0 \leq j \leq J\}$ solves \eqref{eq:time-discrete} while $\{ Y^j;\, 0 \leq j \leq J\}$ solves \eqref{eq:finite-element-discretization}.
\end{thm}
%%%%%%%%%%%%%%%%%%%%%%%%%%%%%%%%%%%%%%%%%%%%%%%%%%%%%%%%%%%%%%%%%%%%%%%%%%%%%%%%%%%%%%%%%%%%%%%%%%%%%%%%%%%%%%%%%%%%%%%%%%%%%%%%%%%%

\begin{proof}
 We subtract \eqref{eq:finite-element-discretization} from \eqref{eq:time-discrete}, and restrict to the test functions $\phi \in {\mathbb V}_h$. Choosing
$\phi = {\mathscr P}_{{\mathbb L}^2} E^j(\omega)$, and using \eqref{defi:l2-projection}, we obtain

\begin{align}
&\frac{1}{2} {\mathbb E}\Bigl[\Vert {\mathscr P}_{{\mathbb L}^2} E^j\Vert^2_{{\mathbb L}^2}
- \Vert {\mathscr P}_{{\mathbb L}^2} E^{j-1}\Vert^2_{{\mathbb L}^2} + 
\Vert {\mathscr P}_{{\mathbb L}^2} [E^j-E^{j-1}]\Vert^2_{{\mathbb L}^2}\Big] \notag \\
& \qquad  + k\mathbb{E}\Big[ \big( \grad E^j, \grad E^j\big)_{\mathbb{L}^2} + 
\big( D\psi(X^j)-D\psi(Y^j), E^j\big)_{\mathbb{L}^2}\Big]  \notag \\
& = k\mathbb{E}\Big[ \big( \grad E^j, \grad( E^j-\mathscr{P}_{\mathbb{L}^2} E^j)\big)_{\mathbb{L}^2} + 
\big( D\psi(X^j)-D\psi(Y^j), E^j-\mathscr{P}_{\mathbb{L}^2} E^j\big)_{\mathbb{L}^2}\Big]  \notag \\
& \qquad + k \mathbb{E}\Big[ \Big( (|X^j|^2-1) \frac{X^j -X^{j-1}}{2} - (|Y^j|^2-1) \frac{Y^j -Y^{j-1}}{2}, \mathscr{P}_{\mathbb{L}^2} E^j\Big)_{\mathbb{L}^2}\Big] \notag \\
&\qquad \quad   +  {\mathbb E}\Bigl[ \Bigl(\sigma(X^{j-1}) -  \sigma(Y^{j-1}), 
 {\mathscr P}_{{\mathbb L}^2} [E^j- E^{j-1}]\Bigr)_{{\mathbb L}^2} \Delta_j W\Bigr] \notag \\
 & \le \frac{k}{2} \mathbb{E}\Big[\|\grad E^j\|_{\mathbb{L}^2}^2\Big] + \frac{k}{2}\mathbb{E}\Big[ \|\grad(X^j
 -\mathscr{P}_{\mathbb{L}^2}X^j)\|_{\mathbb{L}^2}^2\Big] - k \mathbb{E}\Big[ \|E^j\|_{\mathbb{L}^2}^2\Big] \notag  \\
 & \quad  + k \mathbb{E}\Big[ \|\mathscr{P}_{\mathbb{L}^2}E^j\|_{\mathbb{L}^2}^2\Big] 
 + \Big|\mathbb{E}\Big[\Big(|X^j|^2X^j-|Y^j|^2Y^j, X^j -\mathscr{P}_{\mathbb{L}^2}X^j\Big)_{\mathbb{L}^2}\Big]\Big| \notag \\
 &\qquad \quad  + k \mathbb{E}\Big[ \Big( (|X^j|^2-1) \frac{X^j -X^{j-1}}{2} - (|Y^j|^2-1) \frac{Y^j -Y^{j-1}}{2}, \mathscr{P}_{\mathbb{L}^2} E^j\Big)_{\mathbb{L}^2}\Big] \notag \\
 &\qquad \qquad + {\mathbb E}\Bigl[ \Bigl(\sigma(X^{j-1}) -  \sigma(Y^{j-1}), 
 {\mathscr P}_{{\mathbb L}^2} [E^j- E^{j-1}]\Bigr)_{{\mathbb L}^2} \Delta_j W\Bigr] \notag.
\end{align} 
Note that the third term on the right-hand side of the first equality reflects that $f(X^j, X^{j-1})$ is a perturbation of $D\psi(X^j)$. 
By the weak monotonicity property \eqref{AD1}, we see that 
\begin{align*}
 \mathbb{E}\Big[\|\grad E^j\|_{\mathbb{L}^2}^2 - \|E^j\|_{\mathbb{L}^2}^2\Big] \le \mathbb{E}\Big[ \big( \grad E^j, \grad E^j\big)_{\mathbb{L}^2} + 
\big( D\psi(X^j)-D\psi(Y^j), E^j\big)_{\mathbb{L}^2}\Big],
\end{align*}
and therefore we arrive at the following inequality
\begin{align}
&\frac{1}{2} {\mathbb E}\Bigl[ \Bigl( \Vert {\mathscr P}_{{\mathbb L}^2} E^j\Vert^2_{{\mathbb L}^2} - \Vert {\mathscr P}_{{\mathbb L}^2} E^{j-1}\Vert^2_{{\mathbb L}^2}\Bigr) + 
\Vert {\mathscr P}_{{\mathbb L}^2} [E^j-E^{j-1}]\Vert^2_{{\mathbb L}^2}  + k \Vert \nabla E^j\Vert^2_{{\mathbb L}^2}\Bigr] \notag \\
& \leq Ck  {\mathbb E}\bigl[\Vert {\mathscr P}_{{\mathbb L}^2} E^j\Vert^2_{{\mathbb L}^2}\bigr] + C k \mathbb{E}\Big[
\|\grad(X^j - {\mathscr P}_{{\mathbb L}^2} X^j)\|_{{\mathbb L}^2}^2\Big]  \notag \\
 & \quad + Ck \Bigl\vert {\mathbb E}\Bigl[\Bigl( \vert X^j\vert^2 X^j - \vert Y^j\vert^2 Y^j, X^j - {\mathscr P}_{{\mathbb L}^2} X^j\Bigr)_{{\mathbb L}^2}\Bigr]\Big| \notag \\
 & \qquad + k \mathbb{E}\Big[ \Big( (|X^j|^2-1) \frac{X^j -X^{j-1}}{2} - (|Y^j|^2-1) \frac{Y^j -Y^{j-1}}{2}, \mathscr{P}_{\mathbb{L}^2} E^j\Big)_{\mathbb{L}^2}\Big] \notag \\
&\qquad \qquad  +  {\mathbb E}\Bigl[ \Bigl(\sigma(X^{j-1}) -  \sigma(Y^{j-1}), {\mathscr P}_{{\mathbb L}^2} [E^j- E^{j-1}]\Bigr)_{{\mathbb L}^2} \Delta_j W\Bigr] \notag  \\
 &\quad =:Ck  {\mathbb E}\bigl[\Vert {\mathscr P}_{{\mathbb L}^2} E^j\Vert^2_{{\mathbb L}^2}\bigr] + 
 \mathbf{B}_{1,j} + \mathbf{B}_{2,j} + \mathbf{B}_{3,j} + \mathbf{B}_{4,j} . \label{inq:error-1} 
\end{align}
Note that, in view of Lemma \ref{lem:moment-time-discrete}, Young's inequality and the embedding ${\mathbb W}^{1,2} \hookrightarrow {\mathbb L}^6$ for $d \leq 3$
\begin{align*}
\mathbb{E}\Big[\|f(X^j, X^{j-1})\|_{\mathbb{L}^2}^2\Big] & \le C \mathbb{E}\Big[\int_{\mathscr{O}} (|X^j|^4 +1)(|X^j|^2 + |X^{j-1}|^2)\,dx\Big] \\
& \le C \mathbb{E}\Big[\int_{\mathscr{O}} \big(|X^j|^6 + |X^{j-1}|^6 + |X^j|^2 + |X^{j-1}|^2\big)\,dx \Big] \\
  &\le C\Big(1+ \sup_{j}\mathbb{E}\big[| \mathcal{J}(X^j)|^8\big]\Big).
\end{align*}
Thus using Lemma \ref{lem:moment-time-discrete} and the estimate above, we see that 
\begin{align}
 & k\sum_{j=1}^J \mathbb{E}\big[\|\Delta X^j\|_{\mathbb{L}^2}^2\big] \notag \\
 & \le k\sum_{j=1}^J \mathbb{E}\big[\|-\Delta X^j + f(X^j, X^{j-1})\|_{\mathbb{L}^2}^2\big] + 
 k\sum_{j=1}^J\mathbb{E}\Big[\|f(X^j, X^{j-1})\|_{\mathbb{L}^2}^2\Big] \le C. \label{esti:sum-l2-laplace-time-discrete}
\end{align}

Let us recall the following well-known properties of $\mathscr{P}_{\mathbb{L}^2}$, see \cite{BS1}
 \begin{align}\label{esti:projection-h}
  \begin{cases}
   \|g-\mathscr{P}_{\mathbb{L}^2} g\|_{\mathbb{L}^2}\le C h \|g\|_{\mathbb{W}^{1,2}} \quad \forall  g \in \mathbb{W}^{1,2}, \\
    \|g-\mathscr{P}_{\mathbb{L}^2} g\|_{\mathbb{L}^2} + h  \|\nabla[g-\mathscr{P}_{\mathbb{L}^2} g]\|_{\mathbb{L}^2}\le C h^2 \|\Delta g\|_{\mathbb{L}^2} \quad \forall g\in \mathbb{W}^{2,2}.
  \end{cases}
 \end{align}
We use \eqref{esti:sum-l2-laplace-time-discrete} and \eqref{esti:projection-h} to infer that
\begin{align}
 \sum_{j=1}^J \mathbf{B}_{1,j} \leq C h^2\sum_{j=1}^J k \mathbb{E}\Big[\|\Delta X^j\|_{\mathbb{L}^2}^2\Big] \le Ch^2.\notag 
 \end{align}
 Next we estimate $\sum_{j=1}^{J} \mathbf{B}_{4,j}$. A simple approximation argument, \eqref{sac-1b}, and \eqref{esti:sum-l2-laplace-time-discrete} together 
 with Young's inequality lead to 
 \begin{align}
  \sum_{j=1}^{J} \mathbf{B}_{4,j} & \le \sum_{j=1}^J \mathbb{E}\Big[ 
  \|\sigma(X^{j-1})-\sigma(Y^{j-1})\|_{\mathbb{L}^2}\|{\mathscr P}_{{\mathbb L}^2} [E^j- E^{j-1}]\|_{\mathbb{L}^2}
  |\Delta_j W|\Big] \notag \\
  & \le \frac{1}{4}\sum_{j=1}^J \mathbb{E}\Big[\|{\mathscr P}_{{\mathbb L}^2} [E^j- E^{j-1}]\|_{\mathbb{L}^2}^2\Big] + C
  \sum_{j=1}^J k \mathbb{E}\Big[\|E^{j-1}\|_{\mathbb{L}^2}^2\Big]\notag \\
  & \le \frac{1}{4}\sum_{j=1}^J \mathbb{E}\Big[\|{\mathscr P}_{{\mathbb L}^2} [E^j- E^{j-1}]\|_{\mathbb{L}^2}^2\Big] 
  + Ck \sum_{j=1}^J \mathbb{E}\Big[ \|\mathscr{P}_{\mathbb{L}^2} E^{j-1}\|_{\mathbb{L}^2}^2 + \|X^{j-1}-\mathscr{P}_{\mathbb{L}^2} X^{j-1}\|_{\mathbb{L}^2}^2\Big]\notag \\
  & \le  \frac{1}{4} \sum_{j=1}^J {\mathbb E}\bigl[\Vert {\mathscr P}_{{\mathbb L}^2}[E^j - E^{j-1}]\Vert^2_{{\mathbb L}^2}\bigr] +
C\Bigl(h^4 + k\sum_{j=1}^J {\mathbb E}\bigl[\Vert {\mathscr P}_{{\mathbb L}^2}E^{j-1}\Vert^2_{{\mathbb L}^2}\bigr]\Bigr).\notag
 \end{align}
We now bound the term $\mathbf{B}_{2,j}$. We use the algebraic formula given before \eqref{esti:III}, the embedding ${\mathbb W}^{1,2}
\hookrightarrow {\mathbb L}^6$ for $d \leq 3$, and a generalized Young's inequality to have
\begin{align}
 \mathbf{J}_{2,j} & \le Ck \mathbb{E}\Big[ \|E^j\|_{\mathbb{L}^6}\big(\|X^j\|_{\mathbb{L}^6}^2 + \|Y^j\|_{\mathbb{L}^6}^2\big)\|X^j -\mathscr{P}_{\mathbb{L}^2}X^j\|_{\mathbb{L}^2}\Big] \notag \\
 & \le Ck \mathbb{E}\Big[ \|E^j\|_{\mathbb{W}^{1,2}}\big(\|X^j\|_{\mathbb{W}^{1,2}}^2 + \|Y^j\|_{\mathbb{W}^{1,2}}^2\big)
 \|X^j -\mathscr{P}_{\mathbb{L}^2}X^j\|_{\mathbb{L}^2}\Big] \notag \\
 & \le \frac{k}{8}\mathbb{E}\Big[\|\grad E^j\|_{\mathbb{L}^2}^2\Big] + \frac{k}{8}\mathbb{E}\Big[ \| E^j\|_{\mathbb{L}^2}^2\Big]
 + Ck\mathbb{E}\Big[ \Big(\|X^j\|_{\mathbb{W}^{1,2}}^4 + \|Y^j\|_{\mathbb{W}^{1,2}}^4\Big)\|X^j -\mathscr{P}_{\mathbb{L}^2}X^j\|_{\mathbb{L}^2}^2\Big] \notag \\
 & =: \frac{k}{8}\mathbb{E}\Big[\|\grad E^j\|_{\mathbb{L}^2}^2\Big] + \mathbf{B}_{2,j}^1 + \mathbf{B}_{2,j}^2. \notag 
\end{align}
Thanks to \eqref{esti:higher-moment-w12-time-discrete} and \eqref{esti:projection-h}, we note that
\begin{align*}
 \sum_{j=0}^J \mathbf{B}_{2,j}^1 \le \sum_{j=0}^J k \mathbb{E}\Big[\|\mathscr{P}_{\mathbb{L}^2}E^j\|_{\mathbb{L}^2}^2\Big] 
 + Ch^2 k\sum_{j=0}^J \mathbb{E}\Big[\| X^j\|_{\mathbb{W}^{1,2}}^2\Big] 
  \le Ch^2 + \sum_{j=0}^J k \mathbb{E}\Big[\|\mathscr{P}_{\mathbb{L}^2}E^j\|_{\mathbb{L}^2}^2\Big].
\end{align*}
We use \eqref{esti:higher-moment-w12-time-discrete}, \eqref{esti:higher-moment-w12-space-time-discrete} and \eqref{esti:projection-h}, together with 
Young's inequality to get
\begin{align}
 \sum_{j=0}^J \mathbf{B}_{2,j}^2 & \le Ch^2 k\sum_{j=0}^J \mathbb{E}\Big[ \big(\|X^j\|_{\mathbb{W}^{1,2}}^4 + \|Y^j\|_{\mathbb{W}^{1,2}}^4\big)
 \|X^j\|_{\mathbb{W}^{1,2}}^2\Big] \notag \\
 & \le Ch^2 k\sum_{j=0}^J \mathbb{E}\Big[ \|X^j\|_{\mathbb{W}^{1,2}}^8 + \|Y^j\|_{\mathbb{W}^{1,2}}^8 + 
 \|X^j\|_{\mathbb{W}^{1,2}}^4\Big] \le Ch^2. \notag 
\end{align}
It remains to bound $\mathbf{B}_{3,j}$. We decompose $\mathbf{B}_{3,j}$ as follows.
\begin{align}
\mathbf{B}_{3,j} & =  \frac{k}{2} \mathbb{E}\Big[\Big( (|X^j|^2-|Y^j|^2)(X^j-X^{j-1}), \mathscr{P}_{\mathbb{L}^2} E^j\Big)_{\mathbb{L}^2}\Big]  \notag \\
&  \quad \quad + \frac{k}{2} \mathbb{E}\Big[\Big( (|Y^j|^2-1)(E^j-E^{j-1}), \mathscr{P}_{\mathbb{L}^2} E^j\Big)_{\mathbb{L}^2}\Big]=: 
 \mathbf{B}_{3,j}^1 + \mathbf{B}_{3,j}^2. \notag 
\end{align}
Thanks to generalized H\"{o}lder's inequality, the $L^q (1\le q\le \infty)$-stability of $\mathscr{P}_{\mathbb{L}^2}$, the embedding 
${\mathbb W}^{1,2} \hookrightarrow {\mathbb L}^6$ for $d \leq 3$, the estimates \eqref{esti:projection-h}, \eqref{esti:higher-moment-w12-space-time-discrete} and 
\eqref{esti:higher-moment-w12-time-discrete}, we have 
\begin{align}
 \mathbf{B}_{3,j}^2& \le Ck \mathbb{E}\Big[ \|E^j-E^{j-1}\|_{\mathbb{L}^2} \|E^j\|_{\mathbb{W}^{1,2}}\||Y^j|^2-1\|_{\mathbb{L}^3}\Big] \notag \\
 & \le Ck \mathbb{E}\Big[\Big( \|\mathscr{P}_{\mathbb{L}^2}(E^j-E^{j-1})\|_{\mathbb{L}^2} + \|X^j-\mathscr{P}_{\mathbb{L}^2} X^j -
 (X^{j-1}-\mathscr{P}_{\mathbb{L}^2} X^{j-1})\|_{\mathbb{L}^2}\Big) \notag \\
 & \hspace{3cm}\times \|E^j\|_{\mathbb{W}^{1,2}}\||Y^j|^2-1\|_{\mathbb{L}^3}\Big] \notag \\
 & \le \frac{k}{16} \mathbb{E}\big[\| E^j\|_{\mathbb{W}^{1,2}}^2\big]
 + Ck \mathbb{E}\Big[ \|\mathscr{P}_{\mathbb{L}^2}(E^j-E^{j-1})\|_{\mathbb{L}^2}^{1+1}\||Y^j|^2-1\|_{\mathbb{L}^3}^2\Big] \notag \\
 & \qquad \quad+ Ck \mathbb{E}\Big[ \|X^j-\mathscr{P}_{\mathbb{L}^2} X^j -
 (X^{j-1}-\mathscr{P}_{\mathbb{L}^2} X^{j-1})\|_{\mathbb{L}^2}^2 \||Y^j|^2-1\|_{\mathbb{L}^3}^2\Big] \notag \\
 & \le \frac{k}{16} \mathbb{E}\big[\|\grad E^j\|_{\mathbb{L}^2}^2\big] + Ck \mathbb{E}\big[\|\mathscr{P}_{\mathbb{L}^2} E^j\|_{\mathbb{L}^2}^2\big] 
 + \frac{1}{8} \mathbb{E}\Big[\|\mathscr{P}_{\mathbb{L}^2}(E^j-E^{j-1})\|_{\mathbb{L}^2}^2\Big] \notag \\
 & \qquad \quad  + Ck^2 \mathbb{E}\Big[\|\mathscr{P}_{\mathbb{L}^2}(E^j-E^{j-1})\|_{\mathbb{L}^2}^2\||Y^j|^2-1\|_{\mathbb{L}^3}^4\Big]
 + Ckh^2 \mathbb{E}\big[\|X^j\|_{\mathbb{W}^{1,2}}^2\big]  \notag \\
  & \qquad \qquad + Ck h^2 \mathbb{E}\Big[ \Big(\|X^j\|_{\mathbb{W}^{1,2}}^2 + \|X^{j-1}\|_{\mathbb{W}^{1,2}}^2\Big) \||Y^j|^2-1\|_{\mathbb{L}^3}^2\Big] \notag \\ 
  & \le \frac{k}{16} \mathbb{E}\big[\|\grad E^j\|_{\mathbb{L}^2}^2\big]  + \frac{1}{8} \mathbb{E}\Big[\|\mathscr{P}_{\mathbb{L}^2}(E^j-E^{j-1})\|_{\mathbb{L}^2}^2\Big] 
  + Ck \mathbb{E}\big[\|\mathscr{P}_{\mathbb{L}^2} E^j\|_{\mathbb{L}^2}^2\big] + Ckh^2  \notag \\
 & \qquad + Ck^2 \mathbb{E}\Big[\|E^j-E^{j-1}\|_{\mathbb{L}^2}^2(1+ \|Y^j\|_{\mathbb{W}^{1,2}}^8)\Big] \notag \\
  & \qquad \qquad + Ck h^2 \mathbb{E}\Big[1+ \|X^j\|_{\mathbb{W}^{1,2}}^4 + \|X^{j-1}\|_{\mathbb{W}^{1,2}}^4 +  \|Y^j\|_{\mathbb{W}^{1,2}}^8\Big] \notag \\ 
  & \le \frac{k}{16} \mathbb{E}\big[\|\grad E^j\|_{\mathbb{L}^2}^2\big]  + \frac{1}{8} \mathbb{E}\Big[\|\mathscr{P}_{\mathbb{L}^2}(E^j-E^{j-1})\|_{\mathbb{L}^2}^2\Big] 
  + Ck \mathbb{E}\big[\|\mathscr{P}_{\mathbb{L}^2} E^j\|_{\mathbb{L}^2}^2\big] 
 + Ck(h^2 + k).\notag 
\end{align}
Next we estimate $\mathbf{B}_{3,j}^1$. We use the generalized H\"{o}lder's inequality, the $L^q (1\le q\le \infty)$-stability of $\mathscr{P}_{\mathbb{L}^2}$, the embedding 
${\mathbb W}^{1,2} \hookrightarrow {\mathbb L}^6$ for $d \leq 3$, Young's inequality, the estimates \eqref{esti:higher-moment-w12-space-time-discrete}, \eqref{esti:projection-h} and 
 \eqref{esti:l2-diff-time-discrete}, \eqref{esti:l2-norm-discrete-functional} and \eqref{esti:higher-moment-w12-time-discrete}, along with 
Lemma \ref{lem:moment-time-discrete} to get
\begin{align}
 \mathbf{B}_{3,j}^1 & \le Ck \mathbb{E}\Big[ \|\mathscr{P}_{\mathbb{L}^2} E^j\|_{\mathbb{L}^6} \|X^j-X^{j-1}\|_{\mathbb{L}^2} 
 \big\||X^j|^2 - |Y^j|^2\big\|_{\mathbb{L}^3} \Big] \notag \\
 & \le  Ck \mathbb{E}\Big[ \| E^j\|_{\mathbb{W}^{1,2}} \|X^j-X^{j-1}\|_{\mathbb{L}^2} 
 \big(\|X^j\|_{\mathbb{L}^6}^2 +\|Y^j\|_{\mathbb{L}^6}^2\big) \Big] \notag \\
 & \le \frac{k}{16} \mathbb{E}\big[\|\grad E^j\|_{\mathbb{L}^2}^2\big] + \frac{k}{16} \mathbb{E}\big[\| E^j\|_{\mathbb{L}^2}^2\big] 
 + Ck\mathbb{E}\Big[\|X^j-X^{j-1}\|_{\mathbb{L}^2}^2\big(\|X^j\|_{\mathbb{L}^6}^2 +\|Y^j\|_{\mathbb{L}^6}^2\big)^2 \Big] \notag \\
 & \le \frac{k}{16} \mathbb{E}\big[\|\grad E^j\|_{\mathbb{L}^2}^2\big] + Ck \mathbb{E}\big[\|\mathscr{P}_{\mathbb{L}^2} E^j\|_{\mathbb{L}^2}^2\big] 
 + Ckh^2 \mathbb{E}\big[\|X^j\|_{\mathbb{W}^{1,2}}^2\big]  \notag \\
 & \hspace{1cm} + Ck^2 \mathbb{E}\Big[\|X^j\|_{\mathbb{W}^{1,2}}^8 + \|Y^j\|_{\mathbb{W}^{1,2}}^8\Big]+ C \mathbb{E}\Big[\|X^j-X^{j-1}\|_{\mathbb{L}^2}^4\Big] \notag \\
 & \le \frac{k}{16} \mathbb{E}\big[\|\grad E^j\|_{\mathbb{L}^2}^2\big] + Ck \mathbb{E}\big[\|\mathscr{P}_{\mathbb{L}^2} E^j\|_{\mathbb{L}^2}^2\big] 
 + Ck(h^2 + k) \notag \\
 & \hspace{2cm}+ C \mathbb{E}\Big[ k^2\mathcal{J}^2(X^{j-1}) + |\Delta_j W|^4 \|X^{j-1}\|_{\mathbb{L}^2}^4\Big] \notag \\
 & \le \frac{k}{16} \mathbb{E}\big[\|\grad E^j\|_{\mathbb{L}^2}^2\big] + Ck \mathbb{E}\big[\|\mathscr{P}_{\mathbb{L}^2} E^j\|_{\mathbb{L}^2}^2\big] + Ck(h^2 + k) \notag \\
 & \hspace{2.5cm}+ C \mathbb{E}\Big[ k^2\mathcal{J}^2(X^{j-1}) + |\Delta_j W|^4 \big( 1+ \mathcal{J}^2(X^{j-1}\big)\Big] \notag \\
 & \le \frac{k}{16} \mathbb{E}\big[\|\grad E^j\|_{\mathbb{L}^2}^2\big] + Ck \mathbb{E}\big[\|\mathscr{P}_{\mathbb{L}^2} E^j\|_{\mathbb{L}^2}^2\big] 
 + Ck(h^2 + k) + C k^2\big( 1+ \mathbb{E}\big[\mathcal{J}^2(X^{j-1})\big]\big) \notag \\
 & \le \frac{k}{16} \mathbb{E}\big[\|\grad E^j\|_{\mathbb{L}^2}^2\big] + Ck \mathbb{E}\big[\|\mathscr{P}_{\mathbb{L}^2} E^j\|_{\mathbb{L}^2}^2\big] 
 + Ck(h^2 + k).\notag 
\end{align}
 Putting things together in \eqref{inq:error-1} and using the discrete Gronwall's lemma (implicit form) then yields 
\begin{align}
 \sup_{0\le j\le J}\mathbb{E}\Big[\|\mathscr{P}_{\mathbb{L}^2}E^j\|_{\mathbb{L}^2}^2\Big] + k \sum_{j=0}^J 
{\mathbb E}\bigl[\|\grad( X^j - Y^j) \|^2_{{\mathbb L}^{2}}\bigr] \leq C(k+h^2).\label{inq:error-2}
\end{align}
Thus, thanks to  \eqref{esti:higher-moment-w12-time-discrete}, \eqref{esti:projection-h} and \eqref{inq:error-2}, we conclude that 
\begin{align}
 &\sup_{0\le j\le J}\mathbb{E}\Big[\|X^j- Y^j\|_{\mathbb{L}^2}^2\Big] + k \sum_{j=0}^J 
{\mathbb E}\bigl[\|\grad( X^j - Y^j) \|^2_{{\mathbb L}^{2}}\bigr] \notag \\
& \le \sup_{0\le j\le J}\mathbb{E}\Big[\|\mathscr{P}_{\mathbb{L}^2}E^j\|_{\mathbb{L}^2}^2\Big] + k \sum_{j=0}^J 
{\mathbb E}\bigl[\|\grad( X^j - Y^j) \|^2_{{\mathbb L}^{2}}\bigr] + Ch^2 \sup_{0\le j\le J} \mathbb{E}\big[\|X^j\|_{\mathbb{W}^{1,2}}^2\big] \notag \\
& \le C(k+h^2). \notag 
\end{align}
This finishes the proof. 
\end{proof}
\subsection{Proof of Main Theorem}
 Let the assumption \ref{A1} hold and $x\in {\mathbb W}^{2,2}_{{\rm per}}$. Then thanks to Theorem \ref{thm:error-discrete in time}, for every
 $\delta >0$, there exist constants
$0 \leq C_{\delta} < \infty$ and $k_1\equiv k_1(T,x)>0$ such that for all $k\le k_1$ sufficiently small
\begin{align}
\sup_{0 \leq j \leq J} {\mathbb E}\bigl[ \Vert X_{t_j} - X^j\Vert^2_{{\mathbb L}^2}\bigr] + k \sum_{j=0}^J 
{\mathbb E}\bigl[\Vert \grad(X_{t_j} - X^j) \Vert^2_{{\mathbb L}^{2}}\bigr] \leq C_{\delta} k^{1-\delta}, \label{esti:final-1}
\end{align}
where $\{X_t;\, t \in [0,T] \}$ solves \eqref{eq:variational-Allen-Cahn} while $\{ X^j;\, 0 \leq j \leq J\}$ solves \eqref{eq:time-discrete}. Again, Theorem \ref{thm:error-estime} asserts that
there exist constants $C>0$, independent of the discretization parameters $h,k >0$ and $k_2\equiv k_2(T,x)>0$ such that for all $k\le k_2$ sufficiently small
\begin{align}
\sup_{0 \leq j \leq J} {\mathbb E}\Bigl[ \| X^j - Y^j\|^2_{{\mathbb L}^2}\Bigr] + k \sum_{j=0}^J 
{\mathbb E}\Bigl[\|\grad( X^j - Y^j) \|^2_{{\mathbb L}^{2}}\Bigr] \leq C\big(k+h^2\big). \label{esti:final-2}
\end{align}
Let $k_0= \min \{ k_1, k_2\}$. Then \eqref{esti:final-1} and \eqref{esti:final-2} hold true for all $k\le k_0$ sufficiently small. We combine \eqref{esti:final-1} and \eqref{esti:final-2}
to conclude the proof of the main theorem. 
%%%%%%%%%%%%%%%%%%%%%%%%%%%%%%%%%%%%%%%%%%%%%%%%%%%%%%%%%%%

\end{document}